\documentclass[a4paper,11pt]{article}
\usepackage[left=3cm,top=2cm,right=3cm,bottom=2cm,nohead]{geometry}

\usepackage{amsmath, amsthm, amssymb}
\usepackage{verbatim,latexsym}
\usepackage{enumerate,subfigure,color}
\usepackage{graphicx,epsf}
\usepackage{afterpage}

\usepackage{authblk}

\def\hat{\widehat}
\def\tilde{\widetilde}

\newcommand{\Hil}{{\mathcal H}}
\newcommand{\Ave}{{\mathcal L}}
\newcommand{\DOm}{\partial\Omega}
\newcommand{\dd}{{\partial}}

\def\HT{{\mathcal H}}
\def\cal{\mathcal}

\def\Om{\Omega}
\def\O{{\mathcal O}}
\def\re{\hbox{Re}\,}
\def\im{\hbox{Im}\,}
\def\Re{\hbox{Re}\,}         
\def\Im{\hbox{Im}\,}
\def\l{\sigma}
\def\p{\partial}

\def\bar{\overline}
\def\H{{\cal H}}

\def \beq {\begin {eqnarray}}
\def \eeq {\end {eqnarray}}
\def \ba {\begin {eqnarray*}}
\def \ea {\end {eqnarray*}}

\DeclareMathOperator{\dbar}{\bar{\partial}}
\DeclareMathOperator{\de}{\partial}
\DeclareMathOperator{\dez}{\de_z}
\DeclareMathOperator{\dezz}{\de_\zeta}
\DeclareMathOperator{\dbarz}{\dbar_z}
\DeclareMathOperator{\dbarzz}{\dbar_\zeta}

\DeclareMathOperator{\dbark}{\dbar_k}

\DeclareMathOperator{\R}{\mathbb{R}}
\DeclareMathOperator{\C}{\mathbb{C}}
\DeclareMathOperator{\D}{\mathbb{D}}
\DeclareMathOperator{\Z}{\mathbb{Z}}

\DeclareMathOperator{\by}{\times}
\DeclareMathOperator{\bndry}{\partial\Omega}

\newtheorem{theorem}{Theorem}[section]
\newtheorem{lemma}[theorem]{Lemma}

\numberwithin{equation}{section}

\title{A Direct Reconstruction Method for \\Anisotropic Electrical Impedance Tomography}
\author{S.J. Hamilton, M. Lassas, and S. Siltanen}

\affil{Department of Mathematics \& Statistics, University of Helsinki}

\date{}
\begin{document}

\maketitle
\begin{abstract}
\noindent
A novel computational, non-iterative and noise-robust reconstruction method is introduced for the planar anisotropic inverse conductivity problem. The method is based on bypassing the unstable step of the reconstruction of the values of the isothermal coordinates on the boundary of the domain.  Non-uniqueness of the inverse problem is dealt with by recovering the unique isotropic conductivity that can be achieved as a deformation of the measured anisotropic conductivity by \emph{isothermal coordinates}.  The method shows how isotropic D-bar reconstruction methods have produced reasonable and informative reconstructions even when used on EIT data known to come from anisotropic media, and when the boundary shape is not known precisely. Furthermore, the results pave the way for regularized anisotropic EIT. Key aspects of the approach involve D-bar methods and inverse scattering theory, complex geometrical optics solutions, and quasi-conformal mapping techniques.   
\end{abstract}

\section{Introduction} 
A novel computational, non-iterative and noise-robust reconstruction method is introduced for the planar anisotropic inverse conductivity problem. The method is an extension of the so-called D-bar reconstruction methods used in isotropic EIT imaging, and is based on bypassing the unstable steps used in earlier anisotropic methods which involve derivatives of a highly unstable map.  Numerical reconstructions from simulated anisotropic Electrical Impedance Tomography data demonstrate that noise-robust images of isotropic, or scalar-valued, conductivities that are distorted versions of the original anisotropic conductivities can be reliably recovered.

\begin{figure}[b!]
\centering
\begin{picture}(280,140)
\put(0,7){\includegraphics[height=130pt]{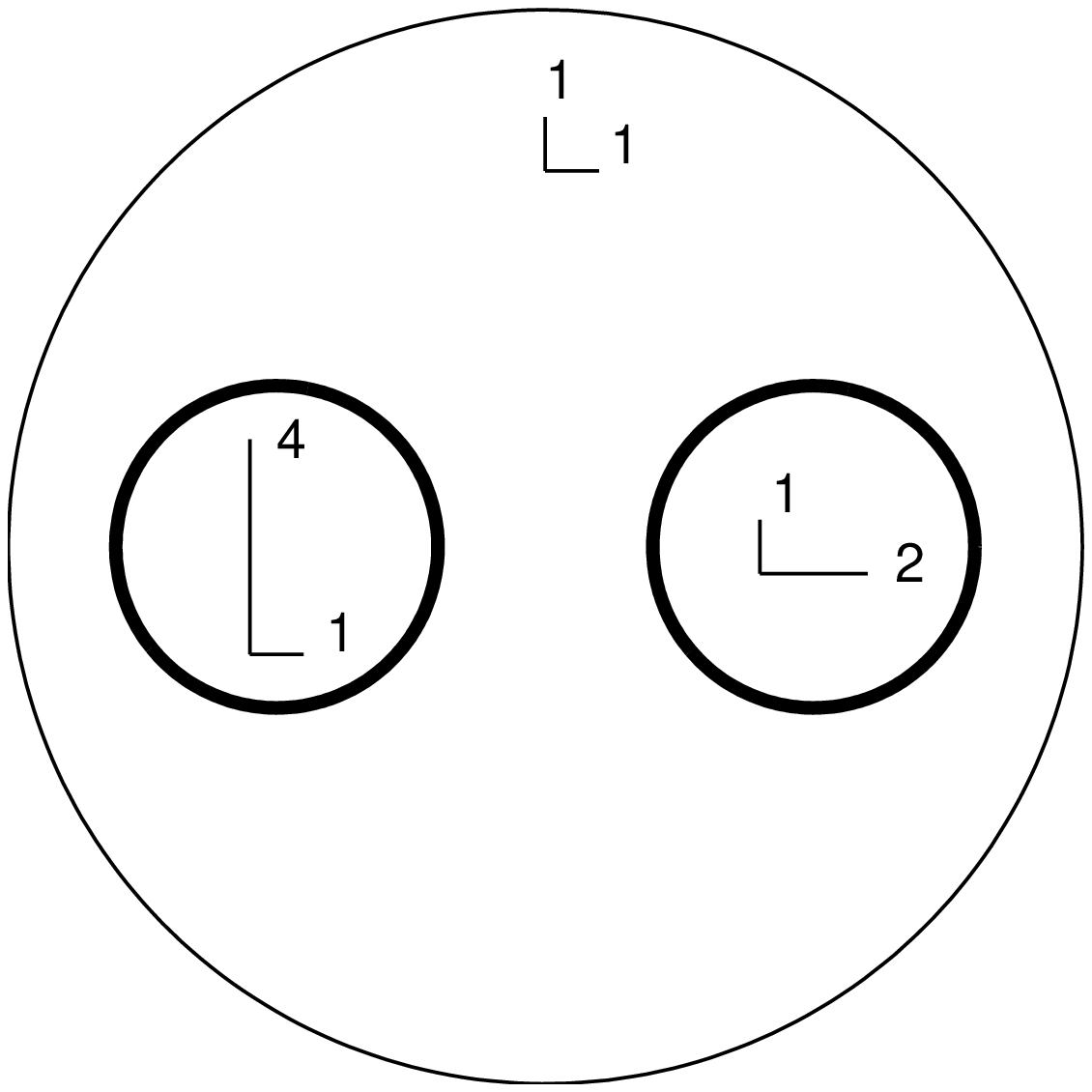}}
\put(150,2){\includegraphics[height=140pt]{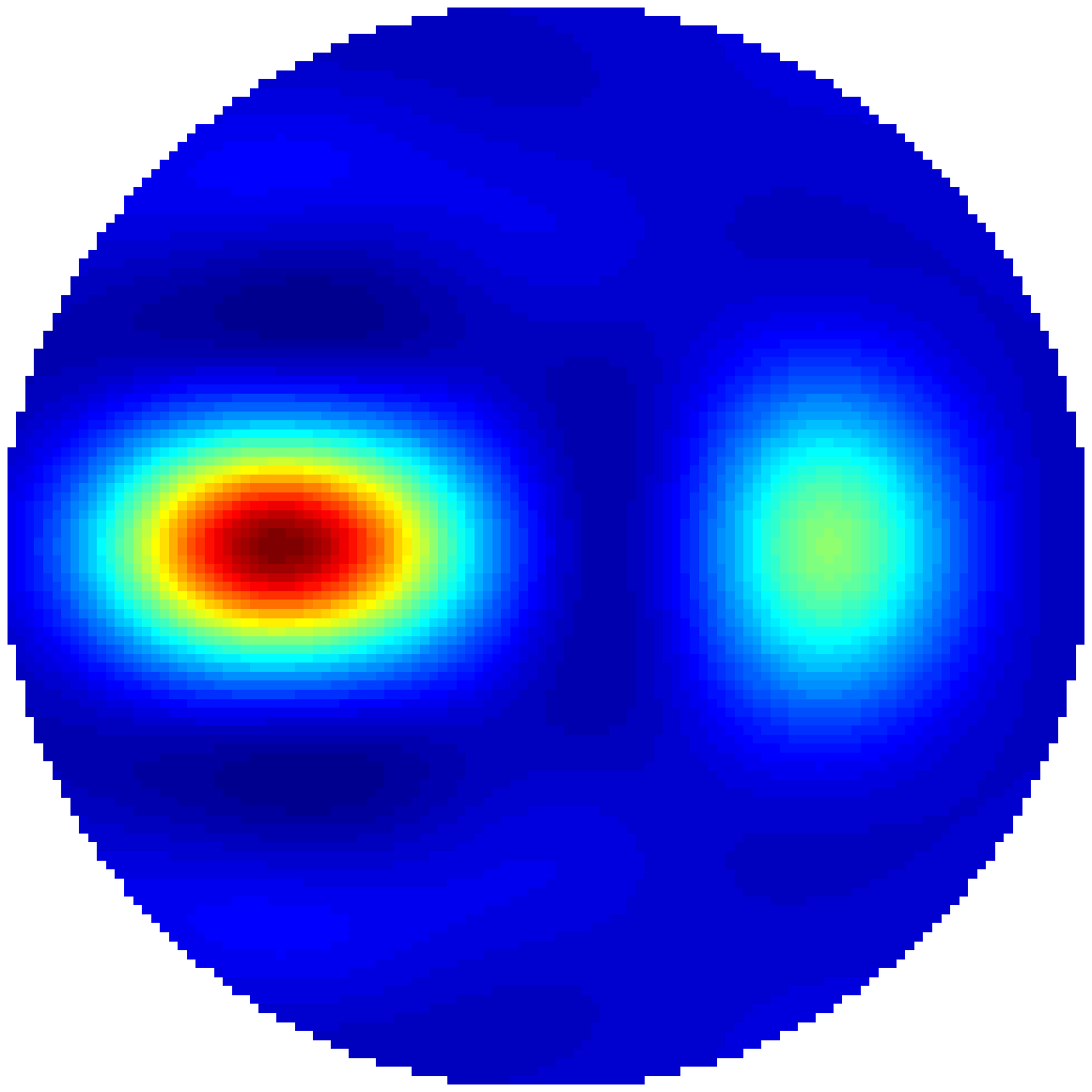}}
\end{picture}
\caption{\label{fig-HammerHead} Left: Anisotropic conductivity $\sigma$ with two circular inclusions in an isotropic background. The strength and direction of the anisotropic preferences are indicated schematically.   Right: Reconstructed conductivity, which is isotropic and appropriately deformed.  The anisotropic preferences squeeze the circular inclusions into ellipses.}
\end{figure}
Electrical Impedance Tomography (EIT) is a portable, inexpensive, non-invasive imaging modality which recovers the internal conductivity of a body using harmless electrical measurements taken at its surface. The reconstruction task is a highly ill-posed mathematical inverse problem whose goal is to produce images that can be used by a physician or engineer for diagnostic and evaluative purposes.  Promising applications of EIT include monitoring heart and lung function in hospitalized patients \cite{Mueller2001,Isaacson2006,Costa2009,Frerichs2000,Holder2005}, detection and classification of breast tumors \cite{Zou2003,Kim2007,Kao2007,Kao2008}, and geophysical prospecting \cite{Abubakar2009}.

The vast majority of reconstruction techniques assume that the medium is isotropic, meaning that the electrical current will flow equally well in all directions. However, many bodies that are imaged  are anisotropic: there is a spatially dependent preferred direction for current flow. For example, human heart tissue is three times more conductive along the muscle fibers than in the lateral directions \cite{Barber1990}. There are relatively few reconstruction approaches designed for anisotropic cases, including \cite{Glidwell1995,Glidwell1997,Abascal2003,Abascal2007,Breckon1992,Lionheart2010}.

Another source of anisotropy is poorly known boundary shape: even if the conductivity is isotropic, incorrect modeling of the domain leads to EIT data that can only arise from an anisotropic conductivity distribution, in the incorrect model domain. The exact boundary shape is often impossible to obtain in practice as even by breathing, or reclining, a patient's chest shape can change significantly. (We remark that there are some techniques available for recovering the boundary shape from EIT data, see \cite{Kolehmainen2005,Kolehmainen2008b,Kolehmainen2010,Kolehmainen2013b,Nissinen2010,Nissinen2011,Darde2013a,Darde2013b}. However, they assume the starting conductivity is isotropic.)

The main reason why the reconstruction literature concentrates on the isotropic case is the {\em non-uniqueness} of the anisotropic problem. In dimension two, deforming the anisotropic conductivity by means of a boundary-fixing diffeomorphism does not change the EIT data, see \cite{Sylvester1990} and \cite[Thm 2]{Nachman1996}. Clearly, there exist uncountably many such deformations. In higher dimensions even less is known.

In our view, non-uniqueness should not be a show-stopper.  Namely, even the isotropic inverse conductivity problem is so severely ill-posed that strong regularization is needed for robust image formation from noisy data. It is well-known from regularization theory \cite{EnglHankeNeubauer1996}, that in linear inverse problems non-uniqueness can be dealt with by picking out a unique representative from the same-data equivalence class, for example the one with minimal square norm. In the nonlinear case of anisotropic EIT, we propose a non-iterative reconstruction method for two-dimensional anisotropic $2\times 2$ conductivities $\sigma$ that recovers the  unique isotropic conductivity that can be achieved as a deformation of $\sigma$ by \emph{isothermal coordinates}. The method can recover useful information in a distorted form, as seen in Figure~\ref{fig-HammerHead}.

Our starting point is the {\em D-bar methodology} for isotropic EIT. Rigorous mathematical analysis of the underlying physical problem has led to the development of direct (non-iterative)  solution techniques, called D-bar methods, which employ inverse scattering theory to uniquely recover the isotropic conductivity.  These methods have proven effective on simulated as well as experimental EIT data \cite{Siltanen2000,Isaacson2004,Isaacson2006,Astala2011,Hamilton2012}, and their  regularization properties are well-understood \cite{Knudsen2009}.  They have even been used on human chest EIT data which is known to correspond to anisotropic tissues.  The resulting images, while sometimes deformed, have nonetheless provided quality information \cite{Isaacson2006}.

We construct a combination of steps appearing in isotropic D-bar methods in such a way that
\begin{enumerate}
\item When applied to data arising from an anisotropic conductivity, the method recovers the above-mentioned unique isotropic conductivity (distorted).
\item When applied to data arising from an isotropic conductivity, the output of the new numerical method coincides with the standard D-bar reconstruction.
\end{enumerate}
Our new method explains how isotropic D-bar reconstruction methods have produced reasonable and informative reconstructions even when used on EIT data known to come from anisotropic media and when the boundary shape is not known precisely. Furthermore, our results pave the way for regularized anisotropic EIT along the lines of the isotropic analysis given in \cite{Knudsen2009}.

The rest of the paper is organized as follows.  Section~\ref{sec-background} provides a brief review of the mathematical history of the isotropic and anisotropic EIT problems and lays out the proposed stable reconstruction method, proven in Section~\ref{sec-CGOproof}.  Section~\ref{sec-algorithm} describes the implementation details for the novel numerical algorithm obtained from the constructive proof given in Section~\ref{sec-CGOproof}.  The algorithm is tested on simulated finite element data for $C^2$ smooth, as well as discontinuous, anisotropic conductivities and the results are presented in Section~\ref{sec-results}.  

\section{Mathematical History of the Problem}\label{sec-background}
\subsection{History of the Isotropic Problem}
The mathematical model for EIT, often called the \emph{inverse conductivity problem} was introduced in Calder\'{o}n's seminal paper \cite{Calder'on1980} for isotropic conductivities $\gamma$.  The voltage potential $u\in H^1(\Omega)$ inside a simply connected domain $\Omega\subset\R^n$ is the unique solution to the elliptic partial differential equation
\begin{equation}\label{eq-cond-iso}
\begin{array}{rcl}
\nabla\cdot\gamma\nabla u &=& \text{0 in $\Omega$}\\
u|_{\bndry} &=& \text{$f$ on $\bndry$},
\end{array}
\end{equation}
with Dirichlet boundary condition defined by the applied voltage $f\in H^{1/2}(\bndry)$, where $\nu$ is the outward facing unit normal vector to $\bndry$.  The domain $\Omega$ is a bounded simply connected set with a smooth boundary $\bndry$ and the conductivity $\gamma:\Omega\to\R$ is a bounded measurable function satisfying $\gamma(x)\geq c>0$ for almost every $x\in\Omega$.  The isotropic inverse conductivity problem is then to recover the scalar-valued coefficient function $\gamma(x)$ from knowledge of the Dirichlet-to-Neumann (D-N) map
\begin{equation}\label{eq-DN-iso}
\Lambda_\gamma: f\mapsto \left.\gamma \frac{\p u}{\p\nu}\right|_{\bndry}.
\end{equation}
Physically, the D-N map is a voltage-to-current density map that describes the current flux at the boundary that results from an applied boundary voltage.  

Calder\'{o}n posited that any bounded isotropic conductivity $\gamma$ could be uniquely determined using measurements taken at the boundary, i.e. the D-N map.  The first major breakthrough towards this goal came in the 1987 paper of Sylvester and Uhlmann where they showed global uniqueness in dimensions 3 and greater for $\C^\infty$ conductivites \cite{Sylvester1987}.  In 1988, Nachman gave the first global constructive proof in dimensions $n\geq 3$ for $\bndry\in C^{1,1}$ and $\gamma\in \C^{1,1}$.  Both papers rely heavily on special solutions called \emph{complex geometrical optics} (\textsc{CGO}) solutions, sometimes called \emph{exponentially growing solutions}.  Such solutions were originally introduced by Faddeev in 1966 \cite{Faddeev1966} and rediscovered by Sylvester and Uhlmann.  The constructive nature of Nachman's proof led to the formation of direct numerical D-bar algorithms by Knudsen \emph{et al.} \cite{Bikowski2011,Knudsen2011,Delbary2011,Bikowski2008a}.  For further information about direct \textsc{CGO} methods in dimensions $n\geq3$  see \cite{Bikowski2011} and the references therein.

Breakthroughs in the two-dimensional problem came in 1996 when Nachman \cite{Nachman1996} presented a constructive \textsc{CGO} based proof using a transformation to the Schr\"odinger equation, and in 2000 when Siltanen \emph{et al.} \cite{Siltanen2000}, using \cite{Nachman1996} as its backbone, developed the first numerical D-bar algorithm for twice-differentiable conductivities from D-N measurements and demonstrated its effectiveness on simulated data.  In the approach, the isotropic conductivity $\gamma$ is recovered by solving a $\dbar$ equation with non-physical scattering data defined by the D-N data.  In 2009, Knudsen \emph{et al.} \cite{Knudsen2009} showed that truncation of the scattering data (i.e. a low pass filtering) corresponds to a regularization strategy.  The numerical D-bar algorithm has also proven effective in the experimental setting using measurements taken on physical electrodes  \cite{Isaacson2004,Isaacson2006,DeAngelo2010}.  We remark that one of the ways to reduce the computation time is to use either a Born approximation or a simpler Green's function to sidestep or simplify the first step in the D-bar algorithm (i.e. solving the boundary integral equation for the traces of the {\sc CGO} solutions).  Such approaches were introduced in \cite{Siltanen2000} and \cite{Mueller2003}, and further examined in \cite{Toufic2012}.

The regularity requirement was reduced to one derivative in 1997 in the constructive \textsc{CGO} proof by Brown and Uhlmann \cite{Brown1997} based on a transformation to a first-order $\de$ and $\dbar$ system (see also \cite{Beals1985}).  In 2003, Knudsen \emph{et al.} \cite{Knudsen2003,Knudsen2004a} completed the proof of Brown and Uhlmann yielding the corresponding D-bar method and tested it on simulated data.  In 2003, Astala and P\"{a}iv\"{a}rinta \cite{Astala2006,Astala2006a} obtained the desired boundedness assumption $\gamma\in L^\infty(\Omega)$, thus solving Calder\'{o}n's problem for planar conductivities, by using a constructive \textsc{CGO} proof based on a transformation to the Beltrami equation.  The corresponding numerical algorithm was implemented in \cite{Astala2010,Astala2011}.

In recent years, the complex isotropic coefficient $\gamma$ problem in 2D has also been an emerging area of study.  In 2000, Francini \cite{Francini2000} extended the constructive results of Brown and Uhlmann to hold for $\gamma$ such that $\Re(\gamma),\; \Im(\gamma) \in W^{2,p}$ for some $p>2$ where $\Im (\gamma)$  is small.  The 2008 proof by Bukhgeim \cite{Bukhgeim2008}, strengthened by Bl{\aa}sten \cite{Eemeli2013PhDThesis}, while not constructive in nature, answers the uniqueness question for smooth and complex 2D $\gamma$ in the affirmative and does not require the smallness condition of \cite{Francini2000}.  Recently, the constructive \textsc{CGO} proof by Francini was completed by Hamilton \emph{et al.} \cite{Hamilton2012} resulting in the first complex coefficient D-bar algorithm which has been demonstrated to work on simulated data \cite{HM12_NonCirc,Hamilton2012,Natalia_Thesis} as well as experimental data. 

The conditional stability of various reconstruction approaches for the inverse (isotropic) conductivity problem have been studied extensively \cite{Liu1997,Barcel'o2001,Barcel'o2007,Clop2010}.

\subsection{History of the Anisotropic Problem}
If the conductivity is anisotropic, i.e. there is a spatially dependent preferred direction for the current to flow, then the conductivity is in fact a matrix-valued function $\sigma=\left[\sigma^{ij}(x)\right]$ where $1\leq i,j\leq n$.  The 2D anisotropic inverse conductivity problem is then to determine the matrix valued coefficient $\sigma(x)$ in the elliptic partial differential equation
\begin{eqnarray}\label{eq-conduct}
\nabla\cdotp \sigma\nabla u= \sum_{i,j=1}^2 
\frac \p{\p x^i}\left[\sigma^{ij}(x)\right] \frac \p{\p x^j} u &=& 0\hbox{ in } \Omega,\\
u|_{\p \Omega}&=&f,\nonumber
\end{eqnarray}
on a simply connected two-dimensional domain $\Omega\subset\R^2$ with prescribed boundary voltage $f\in\H^{1/2}(\bndry)$.  

It is well known that for $\sigma(x)=[\sigma^{ij}(x)]$, $1\leq i,j\leq 2$ a symmetric, positive-definite matrix function, that \eqref{eq-conduct} has a unique solution $u\in H^{1}(\Omega)$.  As $\sigma(x)$ is a symmetric positive-definite matrix, it can be written in terms of an eigen-decomposition $\sigma(x) = P^{-1}(x) D(x)P(x)$ where $P$ is an orthogonal matrix whose rows are the eigenvectors of $\sigma$ and $D$ is a diagonal matrix comprised of the eigenvalues of $\sigma$.  As $\sigma$ is positive definite, the eigenvalues must be positive real values (a reality in realistic conductivities).  In these terms, it is easy to see that the entries of $\sigma$ can be expressed in the traditional basis $P=I$, the $2\by2$ identity matrix, for $\R^2$ where the entries of $D$ then express the preferred direction and extent of the anisotropy of the tissue at a given point.  

As in the isotropic case, one aims to recover the anisotropic conductivity from measurements taken only at the boundary.  When $\sigma$ and $\bndry$ have some smoothness, the voltage-to-current density, or Dirichlet-to-Neumann (D-N), map can be defined as
\begin{equation}\label{eq-DNmap}
\Lambda_\sigma(f)= Bu|_{\p \Omega}=\nu \cdotp \sigma \nabla u\big|_{\bndry},
\end{equation}
where $\nu$ is the outward facing unit normal vector to $\bndry$, $u\in H^1(\Om) $ is the solution of \eqref{eq-conduct}, and $f$ the prescribed voltage on $\bndry$.  
Furthermore, by applying the Divergence theorem one finds 
\begin{equation}\label{eq-Q_l}
Q_{\l,\Om} (\phi):=\int_\Omega  \sum_{i,j=1}^2\sigma^{ij}(x)\frac {\p u}{\p x^i}
\frac {\p u}{\p x^j}
dx=\int_{\p \Omega} \Lambda_\l(\phi) \phi\, dS(x),
\end{equation}
which represents the power needed to maintain the voltage potential $\phi$ on $\bndry$ where $dS(x)$ denotes the arc length on $\bndry$.  By the symmetry
of $\Lambda_\sigma$, knowing $Q_{\l,\Om}$ is equivalent to knowing $\Lambda_\l$.   

Consider a diffeomorphism $F:\Omega\to \Omega$,  $F(x)=(F^1(x),F^2(x))$,  with $F|_{\p
\Omega}=\hbox{Identity}$. Making the change of variables $y=F(x)$
and setting $v=u\circ F^{-1}$ in the first
integral in \eqref{eq-Q_l}, we obtain
\[\nabla\cdotp (F_*\sigma)\nabla v=0\quad\hbox{in }\Om,\]
where
\begin{equation}\label{eq-Fstar-entries}
(F_*\sigma)^{ij}(y)=\left.
\frac 1{\det\left[\frac {\p F^i}{\p x^j}(x)\right]}
\sum_{p,q=1}^2 \frac {\p F^i}{\p x^p}(x)
\,\frac {\p F^j}{\p x^q}(x)  \sigma^{pq}(x)\right|_{x=F^{-1}(y)}.
\end{equation}
where $F_*$ denotes the \emph{push-forward} of the conductivity $\l$ by the diffeomorphism $F$.  Moreover, since $F$ is the identity at $\bndry$, \eqref{eq-Q_l} implies
\[\Lambda_{F_*\l}=\Lambda_\l.\]
Thus, the change of coordinates shows that there is a large class of conductivities which give rise to the same electrical measurements at the boundary, illustrating the non-uniqueness of the 2D anisotropic inverse problem.

The problem is thus not to recover the true anisotropic conductivity (as that is not possible) but rather to determine it up to the action of a class of diffeomorphisms. 
Sylvester \cite{Sylvester1990}, Sun and Uhlmann \cite{Sun2003}, and Astala \emph{et al.} \cite{Astala2005} showed
that there exists a unique isotropic conductivity $\gamma=\sqrt{\det\sigma}$, corresponding to the anisotropic conductivity $\sigma$ under a quasiconformal change of coordinates $F$ (not necessarily preserving the boundary), such that
\[\Lambda_\gamma = \Lambda_{F_*\sigma},\]
for $\sigma\in C^3$, $\sigma\in W^{1,p}$, and $\sigma\in L^\infty$, respectively.  Thus by solving the inverse anisotropic conductivity problem one can recover an isotropic representative for an equivalence class of anisotropic conductivities, i.e. an isotropization of the original conductivity.

In 2005, Astala \emph{et al.} \cite{Astala2005} showed that it is possible to determine an $L^\infty$ smooth anisotropic conductivity up to a $W^{1,2}$ diffeomorphism.  Their result, while constructive in nature, requires the construction of the diffeomorphism $F$ to determine the isotropic D-N map $\Lambda_\gamma$ corresponding to the anisotropic D-N map $\Lambda_\sigma$, and then proceeds along the isotropic $L^\infty$ approach of \cite{Astala2006a}.  In particular, the \emph{push forward} of the D-N map, i.e. $\left(F_{_{\bndry}}\right)_\ast\Lambda_\sigma$ where $F_{_{\bndry}}=F|_{\bndry}:\bndry\to\p\widetilde{\Omega}$, is constructed via
\begin{equation}\label{eq-F-DN}
\left(F_{_{\bndry}}\right)_\ast\Lambda_\sigma:\tilde{h}\mapsto 
\left.\frac{d F^{-1}_{_{\bndry}}}{d\tilde{x}}\middle(\cdot\right)\left[\Lambda_\sigma\left(\tilde{h}\circ F_{_{\bndry}}\right)\right]\circ \left(F_{_{\bndry}}\right)^{-1}
\end{equation}
for functions $\tilde{h}\in H^{1/2}(\p\widetilde{\Omega})$.  Even the determination of $F_{_{\bndry}}$ requires a large frequency limit of the anisotropic traces (see \eqref{eq-Fz-lim}), which from a numerical standpoint, is very unstable, especially in the presence of noise.

In dimensions $n\geq 3$ even less is known.  Kohn and Vogelius \cite{Kohn1984,Kohn1985} showed that only piecewise analytic conductivities can be reconstructed, and for Riemannian manifolds the techniques have been generalized in \cite{Lee1989,Lassas2001,Lassas2003}.

\subsection{The Proposed Approach}
In this paper we build upon the theoretical anisotropic results of \cite{Astala2005} and constructive isotropic proofs of \cite{Astala2006a,Astala2006} and \cite{Nachman1996} to obtain a stable constructive \textsc{CGO} based proof, and related numerical reconstruction algorithm, for anisotropic $C^2$ conductivities that does not construct the unstable diffeomorphism $F$.  We avoid the explicit construction of the map $F$ by instead deriving a formula for the Beltrami scattering data (of \cite{Astala2006a}) in terms of the anisotropic \textsc{CGO} boundary traces (of \cite{Astala2005}).

The reconstruction method from infinite-precision anisotropic D-N data consists of the following steps:

\small
\[\begin{array}{c}
\text{Anisotropic} \\
\text{D-N data}\\
\Lambda_\sigma
\end{array}  \overset{1}{\longrightarrow} \begin{array}{c}
\text{Anistropic}\\
\text{\textsc{CGO} traces}
\end{array}\overset{2}{\longrightarrow} \begin{array}{c}
\text{Beltrami}\\
\text{Scattering}\\
\text{Data}
\end{array} \overset{3}{\longrightarrow}
\begin{array}{c}
\text{Schr\"odinger}\\
\text{Scattering}\\
\text{Data}
\end{array}
\overset{4}{\longrightarrow}
\begin{array}{c}
\text{Isotropic}\\
\text{Representative}\\
\gamma
\end{array}.\]

\normalsize

We remark that Henkin and Santacesaria \cite{HenkinSantacesaria2010}, have proposed an analogous reconstruction method based on
Nachman's formulation, \cite{Nachman1996}.  Our approach is based on the Astala-P\"aiv\"arinta approach, and allows a lower level of regularity for the computation of the scattering data in the reconstruction method.  We then relate the isotropic (Beltrami) scattering data (for $L^\infty$ conductivities) to the isotropic (Schr\"odinger) scattering data and the representative isotropic $W^{2,p}$-smooth conductivity $\gamma$ is recovered using the D-bar method \cite{Siltanen2000} based on Nachman's constructive proof \cite{Nachman1996}.

In Steps 1-3 above, it is enough to require that $\sigma(z)$ and $\widehat{\sigma}=\frac{\sigma(z)}{\det{\sigma(z)}}$
are in $L^\infty(\Omega)$ as the construction is based on Astala-P\"aiv\"arinta
technique \cite{Astala2006a,Astala2006,Astala2005}.  The method proposed by Henkin and Santacesaria \cite{HenkinSantacesaria2010} results in a construction analogous to the Steps 1-3 for an anisotropic conductivity  $\sigma$ that is in  $C^2(\overline \Omega)$ and is isotropic near $\partial \Omega$.  By instead using the Astala-P\"aiv\"arinta approach for these steps, these regularity requirements are improved for the recovery of the scattering data.

In Step 4, numerical studies \cite{AstalaPaivarintaReyesSiltanen2013} indicate that discontinuous isotropic conductivities $\gamma$ can be recovered when connecting the Beltrami and Schr\"odinger scattering data and solving the $\dbar_k$ equation of \cite{Nachman1996}.  However, the theory is not yet complete.  Therefore, in Step~4, we will assume that $\sigma$ is in $W^{2,p}(\Omega)$, $p>1$, and is the constant 1 near $\partial \Omega$.

We note that the assumptions that $\sigma$ is equal to the constant 1 near $\p \Omega$, and $\Omega=\D$ can be removed when conductivity and the boundary of the domain are assumed to be smoother.   Indeed,
if  $\sigma$ is in  $C^3(\overline \Omega)$ and the boundary $\partial \Omega$ is $C^3$-smooth,  the functions equivalent to the boundary values of the determinant
$\det(\sigma)|_{\partial \Omega}$ and its conormal derivative
$\nu\,\cdotp \sigma  \nabla \det(\sigma)|_{\partial \Omega}$
can be uniquely determined using $\Lambda_\sigma$ via explicit formulas
(see \cite{Kang2003}).
Thus, if  $\sigma_0$ is an anisotropic conductivity
in $\Omega_0\subset\subset \Omega$ that is
$C^3(\overline \Omega_0)$-smooth,  $\partial \Omega_0$  is $C^3$-smooth, and we are given
$\partial \Omega_0$ and $\Lambda_{\sigma_0}$, the construction
can be started by reconstructing the functions
$A_0=\det(\sigma_0)|_{\partial \Omega_0}$ and
$B_0=\nu\,\cdotp \sigma_0  \nabla \det(\sigma_0)|_{\partial \Omega_0}.$
We can then construct, in $\Omega\setminus \Omega_0$, a $W^{2,p}(\Omega\setminus \Omega_0)$-smooth, $p>1$, isotropic conductivity $\sigma_1$  such that $\sigma_1$ is one near $\p \Omega$ and
$\det(\sigma_1)|_{\partial \Omega_0}=A_0$ and $\nu\,\cdotp \sigma_1  \nabla \det(\sigma_1)|_{\partial \Omega_0}=B_0$.
We then define the conductivity $\sigma$ in $\Omega$ which
coincides with $\sigma_0$ in $\Omega_0$ and $\sigma_1$ in
$\Omega\setminus \Omega_0$. This conductivity may be only $L^\infty$-smooth
near  $\partial\Omega_0$.  However, this regularity is enough in Steps~1-3 above.  One can
compute the D-N map $\Lambda_\sigma$ using  $\Lambda_{\sigma_0}$ and
$\sigma|_{\Omega\setminus \Omega_0}$ (see \cite{SiltanenTamminen,Nakamura2005}).
Here, we may assume that the domain $\Omega$  is
a larger disc containing the original domain $\Omega_0$.
Numerically, this step involves some FEM solutions of intermediate D-N maps in the intermediate region, but such computations are less demanding than the solution of the D-bar equation.

Finally, the conductivity $\sigma$ can be considered as a Riemannian metric in $\Omega$,  see \cite{Lee1989,Lassas2001,Lassas2003}.  Then, $\sigma$ is piecewisely $C^3$-smooth in the Euclidean coordinates, we see that in the  Riemannian boundary normal coordinates near $\p\Omega_0$ it is $W^{2,p}$-smooth for all $p>1$. Furthermore, when $F:\R^2\to \R^2$ is the quasiconformal map that makes $\sigma$ isotropic, we see that $F_*\sigma$ is an isotropic conductivity that is in $W^{2,p}_{loc}(\R^2)$, and hence it has the regularity needed in Step~4.

Summarizing, the assumptions that $\sigma\equiv 1$ near $\bndry$ and $\Omega=\D$ could in theory be removed. However, for simplicity, we require below that these assumptions
are valid.
\section{The Constructive \textsc{CGO} Anisotropic Proof}\label{sec-CGOproof}
\subsection{Preliminary Considerations}\label{sec-prelim}
In the following we identify $\R^2$ and  $\C$  by the map $(x_1,x_2)\mapsto x_1+ix_2$ and denote $z=x_1+ix_2$. We use the standard notations
\begin{eqnarray*}
\p_z=\frac 12(\p_1-i\p_2),\quad \p_{\overline z}=\frac 12(\p_1+i\p_2), 
\end{eqnarray*}
where  $\p_j=\p/\p x_j$.  For convenience, we work on the unit disc $\Omega=\D\subset\R^2$ but the results generalize to arbitrary domains.  Below we consider $\sigma:\Om\to \R^{2\times 2}$, with $\sigma=1$ near $\bndry$,  to be extended as a function $\sigma:\C\to \R^{2\times 2}$ by defining $\sigma(z)=I$ for  $z\in \C\setminus \Om$. 
If $\Om\subset \R^2$ is a bounded domain, it is convenient
to consider the class of  matrix functions $\sigma=[\sigma^{ij}]$
such that 
\begin{equation}\label{basic assump}
[\sigma^{ij}]\in L^\infty(\Om;\R^{2\times 2}), \quad [\sigma^{ij}]^t=[\sigma^{ij}],\quad C_0^{-1}I \leq [\sigma^{ij}]\leq C_0I,
\end{equation}
where $C_0>0$ and the super-script $^t$ denotes the non-conjugate matrix transpose.  In this paper, the minimal possible value of $C_0$ is denoted by $C_0(\sigma)$.  Note that it is necessary to require $C_0(\sigma)<\infty$, else there would be counterexamples showing that even the equivalence class of the conductivity cannot be recovered \cite{Greenleaf2003,Greenleaf2003a}.

For the symmetric positive-definite conductivity $\sigma=\sigma^{ij}$ we define the coefficient (see \cite{Sylvester1990,Astala2006a,Iwaniec2001})
\begin{eqnarray}\label{eq-mu-tilde}
\widetilde{\mu}(z)= 
\frac {\sigma^{11}(z)-\sigma^{22}(z)+2i\,\sigma^{12}(z)}
 {\sigma^{11}(z)+\sigma^{22}(z)+2\sqrt{\det(\sigma(z))}},
\end{eqnarray}
which will serve as the Beltrami coefficient for the change of coordinates described below.  The coefficient $\widetilde{\mu}(z)$ satisfies $|\widetilde{\mu}(z)|\leq \kappa<1$ and is compactly supported.

Next we introduce a $W^{1,2}$-diffeomorphism (not necessarily preserving the boundary) that transforms the anisotropic conductivity into an isotropic one. 

\begin{lemma}\label{lem: 1} 
There is a quasiconformal homeomorphism $F:\C\to \C$
such that
\begin{equation}\label{eq-asympt-withA}
F(z)=z+\frac A{z}+\O\left(\frac 1{z^2}\right)=z+\O\left(\frac 1{z}\right)\quad\hbox{as }|z|\to \infty,
\end{equation}
and such that $F\in W^{1,p}_{loc}(\C;\C)$, $2<p<p(C_0)=\frac {2C_0}{C_0-1}$
for which
\begin{equation}\label{eq-isotropic-gam-from-sig-F}
(F_*\sigma)(z)=\gamma(z):=
\det\sqrt{\sigma(F^{-1}(z))}.
\end{equation}
\end{lemma}

\begin{proof}
The existence of such an $F$ with $F(z)=z+\O(1/z)$ is given in \cite{Astala2005}.  Here we make the asymptotics of $F$ more precise.  

The map $F$ satisfies the following Beltrami equation
\begin{equation}\label{eq-Beltrami-F}
\dbarz F(z)=\widetilde{\mu}(z)\dez F(z), \quad z\in\C
\end{equation}
where $\tilde{\mu}$ is given in \eqref{eq-mu-tilde}.  Define 
\begin{equation}\label{eq-Gz-lem1}
G(z)=\frac{1}{F\left(\frac{1}{z}\right)-F(0)}.
\end{equation}
As $F:\C\to\C$ is a homeomorphism, $F(z)=z+\O(1/z)$, and $\dbarz F(z)=0$ for $|z|>1$, $G(z)$ satisfies 
\begin{eqnarray*}
\dbarz G(z)&=& 0 \quad \text{in $\D\setminus\{0\}$}\\
\lim_{z\to 0} G(z) &=& \lim_{z\to 0} \frac{z}{1+\O(1)-zF(0)}=0.
\end{eqnarray*}
Therefore $G$ has a removable singularity at $z=0$, and if we define $G(0)=0$, then the extended $G$ satisfies 
\[\dbarz G(z)= 0 \quad \text{in $\D$},\]
and the expansion
\[G(z)=z+b_2z^2+b_3 z^3+\cdots\]
holds near $z=0$.  Furthermore, $H(z):=z/G(z)$ is analytic in $B_\epsilon(0)$ as it is analytic in $B_\epsilon(0)\setminus\{0\}$ and bounded as $\lim_{z\to 0}\frac{z}{G(z)}=1$, and thus
\[H(z)=\sum_{j=0}^\infty c_j z^j \quad \text{for $|z|<\epsilon$ and $c_0=1$.}\]
Now, from \eqref{eq-Gz-lem1}, 
\begin{eqnarray*}
F(z)&=&z\;H\left(\frac{1}{z}\right)+F(0)\\
&=& z\left(1+c_1\frac{1}{z} + c_2\frac{1}{z^2} +c_3\frac{1}{z^3}\right) + F(0)\\
&=& \left(c_1+F(0)\right) + z+ c_2\frac{1}{z} + c_3\frac{1}{z^2} + \cdots\\
&=& z+\frac{A}{z} + \O\left(\frac{1}{z^2}\right)
\end{eqnarray*}
for $|z|>\frac{1}{\epsilon}$, where $c_1=-F(0)$ is dictated by the asymptotics of $F(z)$, and $c_2=A$ is the sought-after constant, completing the proof.  
\end{proof}

Next we review the setup for the anisotropic problem described in \cite{Astala2005}.  In addition to the anisotropic conductivity equation \eqref{eq-conduct}, we consider the corresponding isotropic conductivity equation for $\gamma(\zeta)=\left(F_*\sigma\right)(\zeta)$ where $\zeta\in\tilde{\Omega}=F(\Omega)$.  For these considerations, we observe that if $u$ satisfies equation \eqref{eq-conduct} and $\gamma$ is as in \eqref{eq-isotropic-gam-from-sig-F} then the function
\[w(\zeta)=u(F^{-1}(\zeta))\in H^1\left(\tilde \Om\right),\]
satisfies the isotropic conductivity equation
\begin{eqnarray}\label{EQ 1}
\nabla\cdotp \gamma\nabla w&=&0\quad \hbox{ in }\tilde \Om,\\
w|_{\p \tilde \Omega}&=&\phi\circ F^{-1}.\nonumber
\end{eqnarray}
Thus, $\gamma$ can be considered as a scalar, isotropic $L^\infty$ smooth
conductivity $\gamma I$. We also extend the isotropic function $\gamma:\tilde \Om\to \R_+$ to a function $\gamma:\C\to \R_+$ by defining $\gamma(\zeta)=1$ for  $\zeta\in \C\setminus \tilde \Om$.

While solving the isotropic inverse problem in \cite{Astala2006a}, the interplay between the scalar conductivities $\gamma(\zeta)$ and $\frac{1}{\gamma(\zeta)}$ played a crucial role.  Motivated by this, we define
\[\hat \sigma^{ij}(z)=\frac 1{\det(\sigma(z))}\sigma^{ij}(z), \quad z\in\C,\]
for the anisotropic conductivity $\sigma$.  Note that for an isotropic conductivity $\sigma$, this would reduce to $\hat \sigma=1/\sigma$.

Now let $F$ be the quasiconformal map defined in  Lemma~\ref{lem: 1} and $\gamma=F_*\sigma$ as in \eqref{eq-isotropic-gam-from-sig-F}.  We say that $\hat w\in H^1\left(\tilde \Om\right)$ is a $\gamma$-harmonic conjugate of $w$ if
\begin{eqnarray}\label{EQ 2}
\p_1\hat w(\zeta)&=&-\gamma(z)\p_2 w(\zeta)\\
\p_2\hat w(\zeta)&=&\gamma(z)\p_1 w(\zeta)\nonumber
\end{eqnarray}
where $\zeta\in\widetilde{\Omega}$.  Using $\hat w$, we define the function $\hat u$, which we call the $\sigma$-harmonic conjugate of $u$,
\[\hat u(z)=\hat w(F(z)).\]
To determine the equation governing $\hat u$, it easily follows that (c.f. \cite{Astala2006a})
\begin{equation}\label{EQ 3}
\nabla\cdotp \frac 1 {\gamma} \nabla\hat w=0\quad \hbox{ in }\tilde \Om,
\end{equation}
and by changing the coordinates to $\zeta=F(z)$ we see that $1/\gamma=F_*\hat \sigma$.
These facts imply
\begin{equation}\label{EQ 4}
\nabla\cdotp \hat \sigma \nabla\hat u=0\quad \hbox{ in }\Om.
\end{equation}
Thus, $\hat u$ is the $\widehat\sigma$-harmonic conjugate function of $u$ and we have 
\begin{equation}\label{EQ 5}
\nabla \hat u=J \sigma \nabla u,\quad
\nabla  u=J\widehat \sigma \nabla \hat u, \qquad \text{where $J=\left(\begin{array}{cc}
0 & -1\\
1 & 0\\
\end{array}\right)$}.
\end{equation}
Since $u$ is a solution to the conductivity equation if and only if $u+c$ (where $c\in \C$) is a solution, we see from (\ref{EQ 5}) that the Cauchy data pairs $C_{\sigma}$ determine the pairs  $C_{\hat \sigma}$ and vice versa. Thus we get, almost for free, that $\Lambda_\sigma$ determines  $\Lambda_{\hat \sigma}$, as well.

Lemma 3.2 and Corollary 3.3 of \cite{Astala2005} describe the relationship between the anisotropic conductivity equations
\[\nabla\cdot\sigma\nabla u=0,\qquad \text{and} \qquad \nabla\cdot\widehat{\sigma}\nabla\widehat{u}=0, \quad z\in\C,\]
and the anisotropic Beltrami equation for $\sigma$, for fixed $k\in\C$,
\begin{equation}\label{eq-aniso-Beltrami}
\begin{array}{rcl}
\dbarz g(z)&=& \mu_1(z)\dez g(z)+\mu_2(z)\overline{\dez g(z)}, \quad z\in\C\\
g(z,k) &=& e^{ikz}\left(1+\O_k\left(\frac{1}{|z|}\right)\right),
\end{array}
\end{equation}
for $g=u+i\widehat{u}$ and the scalar Beltrami coefficients $\mu_1$ and $\mu_2$ defined by
\begin{equation}\label{eq-aniso-Beltrami-coeffs-mu1-mu2}
\mu_1=\frac{\sigma^{22}-\sigma^{11}-2i\sigma^{12}}{1+\text{tr}(\sigma)+\det(\sigma)},\qquad \text{and} \qquad \mu_2=\frac{1-\det(\sigma)}{1+\text{tr}(\sigma)+\det(\sigma)}.
\end{equation}
Note that $\O_k(h(z))$ denotes a function of $(z,k)$ that satisfies
\[\left|\O_k(h(z))\right|\leq C(k)|h(z)|,\quad \forall z\in\C,\]
where $C(k)$ is a constant depending only on $k\in\C$.
We will also make use of the Beltrami system for $\hat{\sigma}$
\begin{equation}\label{eq-aniso-Beltrami-Minusmu2}
\begin{array}{rcl}
\dbarz \widetilde{g}(z)&=& \mu_1(z)\dez \widetilde{g}(z)-\mu_2(z)\overline{\dez\widetilde{g}(z)}, \quad z\in\C\\
\widetilde{g}(z,k) &=& e^{ikz}\left(1+\O_k\left(\frac{1}{|z|}\right)\right),
\end{array}
\end{equation}
resulting from using $\widehat{\sigma}$ in \eqref{eq-aniso-Beltrami}.  Here $\mu_1$ and $\mu_2$ are as defined above in \eqref{eq-aniso-Beltrami-coeffs-mu1-mu2} in terms of $\sigma$.

From \cite{Astala2005}, the boundary values of $u$ and $\widehat{u}$ are related by
\[\tau\cdot\nabla\widehat{u}|_{\bndry}=\Lambda_\sigma\left({u}\middle|_{\bndry}\right),\]
or equivalently,
\begin{equation}\label{eq-u-uhat-DN-tang}
\partial_T\widehat{u}=\Lambda_\sigma\left({u}\right), \quad \text{for $z\in\bndry$},
\end{equation} 
where $T=(-\nu_2,\nu_1)$ is the unit tangent vector to $\bndry$ and thus $\partial_T$ denotes the tangential derivative on $\bndry$ in the counter-clockwise direction.  We define the $\sigma$-Hilbert transform by
\begin{eqnarray}\label{eq-H-sig}
\H_\sigma&:& H^{1/2}(\p \Om)\to H^{1/2}(\p \Om)/\C\\
\nonumber
& &\re g|_{\p \Om}\mapsto \im g|_{\p \Om}+\C.
\end{eqnarray}
Note that if $\re g|_{\p \Om}=u|_{\p \Om}$ and  $\im g|_{\p \Om}=\hat u|_{\p \Om}$, then the $\sigma$-Hilbert transform $\H_\sigma$ connects the function $u$ and it's $\widehat{\sigma}$-conjugate function $\widehat{u}$ by
\begin{equation}\label{eq-Hmu-def-aniso}
\mathcal{H}_\sigma (u)=\widehat{u}.
\end{equation}
By taking the tangential derivative and \eqref{eq-u-uhat-DN-tang} we have
\begin{equation}\label{eq-Hmu-def-aniso-tan-deriv}
\partial_T\left(\mathcal{H}_\sigma (u)\right)=\Lambda_\sigma u,
\end{equation}
and thus, as in \cite{Astala2006}, we see that $\Lambda_\sigma$ determines the $\sigma$-Hilbert transform $\H_\sigma$ defined by \eqref{eq-H-sig} and the $\sigma$-Hilbert transform is essentially just a reformulation of the D-N map $\Lambda_\sigma$.

From \cite[Lemma 3.4]{Astala2005} we have the following lemma involving the anisotropic D-N maps $\Lambda_\sigma$ and $\Lambda_{\hat\sigma}$:
\begin{lemma}\label{lem: 2} 
The Dirichlet-to-Neumann map $\Lambda_\sigma$
determines the maps $\Lambda_{\hat \sigma}$ and $\H_\sigma$.
\end{lemma}

\subsection{The Non-linear Isotropic Beltrami Scattering Transform}\label{sec-scat}
Next we review the scattering data for the $\dbark$ equation related to the isotropic Beltrami equation \cite{Astala2006,Astala2006a}, in our setting.  Let 
\begin{equation}\label{eq-u-gamma}
u_\gamma(\zeta,k)=\re f_{\mu_\gamma}(\zeta,k)+i\,\im f_{-\mu_\gamma}(\zeta,k),
\end{equation}
where $f_{\pm\mu_\gamma}$ are the \textsc{CGO} solutions to the following isotropic Beltrami equations, for fixed $k\in\C$,
\begin{equation}\label{eq-u-gamma-beltrami}
\begin{array}{rcl}
\dbarzz f_{\pm\mu_\gamma} (\zeta,k)& = & \pm\mu_\gamma(\zeta)\overline{\dezz f_{\pm\mu_\gamma}(\zeta,k)}, \quad \zeta\in\C\\
f_{\pm\mu_\gamma}(\zeta,k)&=& e^{ik\zeta}\left(1+\O_k\left(\frac{1}{|\zeta|}\right)\right),
\end{array}
\end{equation}
where
\begin{equation}\label{eq-mu-gamma}
\mu_\gamma(\zeta) = \frac{1-\gamma(\zeta)}{1+\gamma(\zeta)},\quad\quad-\mu_\gamma(\zeta) = \frac{1-1/\gamma(\zeta)}{1+1/\gamma(\zeta)},
\end{equation}
are the corresponding Beltrami coefficients for the isotropic conductivity $\gamma(\zeta)=(F_*\sigma)(\zeta)$ and $1/\gamma(\zeta)=(F_*\widehat{\sigma})(\zeta)$, respectively.  Here $\sigma$ denotes the original anisotropic conductivity, and $F:\C\to\C$ the unique diffeomorphism whose \emph{push-forward} of $\sigma(z)$ is the isotropic conductivity $\gamma(\zeta)$, with asymptotic condition as in Lemma~\ref{lem: 1}.

As $\gamma(\zeta)\equiv 1$ for $\zeta\in\C\setminus\widetilde{\Omega}$, we have $\mu_\gamma(\zeta)=0$ and thus $f_{\pm\mu_\gamma}(\zeta,k)$ are analytic in $\zeta$ for $\zeta\in\C\setminus\widetilde{\Omega}$.  Writing
\begin{equation}\label{eq-f-M-decomp}
f_{\pm\mu_\gamma}(\zeta,k)=e^{ik\zeta}M^\pm_\gamma(\zeta,k),
\end{equation}
where 
\begin{equation}\label{eq-f-M-asymptotics}
M^\pm_\gamma(\zeta,k)  = 1+\O_k\left(\frac{1}{|\zeta|}\right),\quad |\zeta|\to\infty,
\end{equation}
we see that $M^\pm_\gamma(\zeta,k)$ are also analytic for $\zeta\in\C\setminus\widetilde{\Omega}$, and thus admit power series representations
\begin{equation}\label{eq-M-expansion}
M^\pm_\gamma(\zeta,k) = \sum_{n=0}^\infty \frac{b^\pm_n(k)}{\zeta^n} 
=  b^\pm_0(k) + \frac{b^\pm_1(k)}{\zeta^1}+\frac{b^\pm_2(k)}{\zeta^2}+\frac{b^\pm_3(k)}{\zeta^3}+\cdots,
\end{equation}
for $|z|>R$ large, and that $b^\pm_0(k)=1$ from the large $|\zeta|$ asymptotics of $M^\pm_\gamma$ given in \eqref{eq-f-M-asymptotics}.  Note that $b^\pm_n(k)$ denote the coefficients in the expansion and depend on the scattering parameter $k\in\C$.

Using this expansion for $M^\pm_\gamma(\zeta,k)$ in the formulation of $f_{\pm\mu_\gamma}(\zeta,k)$, we see
\begin{eqnarray}
f_{\pm\mu_\gamma}(\zeta,k) &=& e^{ik\zeta}\left(1+\sum_{n=1}^\infty \frac{b^\pm_n(k)}{\zeta^n}\right) \label{eq-f-expansion}\\
&=&  e^{ik\zeta}+ \frac{b^\pm_1(k)}{\zeta^1}e^{ik\zeta}+\frac{b^\pm_2(k)}{\zeta^2}e^{ik\zeta}+\frac{b^\pm_3(k)}{\zeta^3}e^{ik\zeta}+\cdots,\nonumber
\end{eqnarray}
for $|\zeta|>R$.  By construction, $u_\gamma(\zeta,k)=\re f_{\mu_\gamma}(\zeta,k)+i\,\im f_{-\mu_\gamma}(\zeta,k)$ is a solution to the isotropic conductivity equation
\begin{equation}\label{eq-isotropic-div-eq}
\nabla\cdot\gamma(\zeta)\nabla u_\gamma(\zeta,k)=0.
\end{equation}

From \cite[(18.34)]{Astala2009}, $u_\gamma$ satisfies the following D-bar equation in the auxiliary variable $k$
\begin{equation}\label{eq-dbark-equation-iso}
\dbark u_\gamma(\zeta,k) = -i\tau(k) \overline{u_\gamma(\zeta,k)}, \quad \zeta\in\C,
\end{equation}
where 
\begin{equation}\label{eq-scat-iso-def}
\tau(k)=\frac{1}{2}\left(\overline{b^+_1(k)}-\overline{b^-_1(k)}\right),
\end{equation}
plays the role of the non-physical scattering data, and $b^\pm_1(k)$ are coefficients in the large $|\zeta|$ expansion \eqref{eq-M-expansion}.

Notice that $\tau(k)$ is independent of the spatial coordinates, only requiring $|\zeta|$ large to compute the coefficients $b^\pm_1(k)$.  This will play an important role in relating the anisotropic and isotropic Beltrami problems.  This scattering data $\tau(k)$, for $L^\infty$ isotropic conductivities, is related to the scattering data $\mathbf{t}(k)$ of \cite{Nachman1996} for $C^2$ smooth conductivities via \cite{Astala2010}
\[\mathbf{t}(k)=-4\pi i\bar{k}\tau(k),\]
which we will use to avoid forming the unstable coordinate deformation map $F$.  Next we define the scattering data $\tau(k)$ in terms of the measured anisotropic boundary data.
\subsubsection{Boundary Integral Equations for the Scattering Data}\label{eq-scat-BIEs}
We now seek to relate the anisotropic and isotropic \textsc{CGO} solutions, in a similar manner as \cite{Astala2005}, in order to evaluate the isotropic Beltrami scattering transform $\tau(k)$ in terms of the anisotropic Beltrami \textsc{CGO}s.  We briefly review the setting here.

We seek \textsc{CGO} solutions $W^+_\sigma(z,k)$ for $k\in\C$ and  $z\in\C\setminus\D$ for the anisotropic ${\sigma}$ problem satisfying
\begin{eqnarray}
\dbarz W^+_\sigma(z,k) &=& 0, \quad \text{for $z\in\C\setminus\D$ }\label{eq-dbarz Wplus-sig}\\
W^+_\sigma(z,k)&=& e^{ikz}\left(1+\O_k\left(\frac{1}{|z|}\right)\right)\label{eq-Wplus-sig-asym}\\
\im \left.W^+_\sigma(z,k)\right|_{z\in\partial\D}&=& \HT_{\sigma}\left(\re \left.W^+_\sigma(z,k)\right|_{z\in\partial\D}
\right).\label{eq-Wplus-sigma-BIE}
\end{eqnarray}
Such solutions correspond to the restriction of the {\sc CGO} solutions $g(z,k)$ of the anisotropic conductivity equation \eqref{eq-aniso-Beltrami} to the exterior domain $\C\setminus\D$, i.e. $W^+_\sigma(\cdot,k)=g(\cdot,k)|_{\C\setminus\D}$.  There $\sigma=I$, the $2\by2$ identity matrix, and thus  $\mu_1(z)=0=\mu_2(z)$ for $|z|\geq 1$ and hence $W^+_\sigma(z,k)$ is harmonic in $\C\setminus\D$.  

For the isotropic conductivity problem, we seek \textsc{CGO} solutions $W^+_\gamma(\zeta,k)$ for $k,\, \zeta\in\C$ satisfying the isotropic Beltrami problem
\begin{eqnarray}
\dbarzz W^+_\gamma(\zeta,k) &=& \mu_\gamma(\zeta)\overline{\dezz W^+_\gamma(\zeta,k)}, \quad \zeta\in\C \label{eq-dbarz Wplus-gam}\\
W^+_\gamma(\zeta,k)&=& e^{ik\zeta}\left(1+\O_k\left(\frac{1}{|\zeta|}\right)\right).\label{eq-Wplus-gam-asym}
\end{eqnarray}

Similarly for the $\widehat{\sigma}$ anisotropic and $1/\gamma$ isotropic problems we desire \textsc{CGO} solutions $W^-_\sigma(z,k)$ and $W^-_\gamma(\zeta,k)$, respectively, such that
\begin{eqnarray}
\dbarz W^-_\sigma(z,k) &=& 0, \quad \text{for $z\in\C\setminus\D$ }\label{eq-dbarz Wminus-sig}\\
W^-_\sigma(z,k)&=& e^{ikz}\left(1+\O_k\left(\frac{1}{|z|}\right)\right)\label{eq-Wminus-sig-asym}\\
\im \left.W^-_\sigma(z,k)\right|_{z\in\partial\D}&=& \HT_{\widehat{\sigma}}\left(\re \left.W^-_{{\sigma}}(z,k)\right|_{z\in\partial\D}
\right),\label{eq-Wminus-sigma-BIE}
\end{eqnarray}
where $W^-_\sigma(z,k)$ is the restriction of $\tilde{g}(z,k)$ from \eqref{eq-aniso-Beltrami-Minusmu2} to the exterior domain $\C\setminus\D$, i.e. $W^-_\sigma(\cdot,k)=g(\cdot,k)|_{\C\setminus\D}$, and
\begin{eqnarray}
\dbarzz W^-_\gamma(\zeta,k) &=& -\mu_\gamma(z)\overline{\dezz W^-_\gamma(\zeta,k)}, \quad \zeta\in\C \label{eq-dbarz Wminus-gam}\\
W^-_\gamma(\zeta,k)&=& e^{ik\zeta}\left(1+\O_k\left(\frac{1}{|\zeta|}\right)\right).\label{eq-Wminus-gam-asym}
\end{eqnarray}

Recall that by the definition of $\mu_\gamma(\zeta,k)=\frac{1-\gamma(\zeta)}{1+\gamma(\zeta)}$, the isotropic \textsc{CGO} solutions $W^\pm_\gamma(\zeta,k)$ are harmonic in $\zeta$ for $\zeta\in\C\setminus\widetilde{\Omega}$.  The \textsc{CGO} solutions to the anisotropic and isotropic problems in the exterior domain are closely related via a change of coordinates.  The following lemma, from \cite[Lemma 3.5]{Astala2005}, clarifies the relation and is stated for $W^+_\sigma$ and $W^+_\gamma$.  The result also holds for $W^-_\sigma$ and $W^-_\gamma$ using the same map $F$.
\begin{lemma}\label{lem-ALP-aniso-iso-CGO-relation}
For all $k\in\C$ we have:
\begin{enumerate}[(i)]
\item The system \eqref{eq-dbarz Wplus-gam}-\eqref{eq-Wplus-gam-asym} has a unique solution $W^+_\gamma(\zeta,k)$ for $\zeta\in\C$.
\item The system \eqref{eq-dbarz Wplus-sig}-\eqref{eq-Wplus-sigma-BIE} has a unique solution $W^+_\sigma(z,k)$ for $z\in\C\setminus\D$.
\item For $z\in\C\setminus\D$, we have
\begin{equation}\label{eq-aniso-equals-iso-exterior}
W^+_\sigma(z,k)=W^+_\gamma(F(z),k),
\end{equation}
where $F:\C\to\C$ denotes the diffeomorphism such that $(F_*\sigma)(z)=\gamma(z)$ described in Lemma~\ref{lem: 1}.
\end{enumerate}
\end{lemma}

Above we saw that the scattering transform $\tau(k)$ is independent of the spatial coordinates.  We will now use Lemma~\ref{lem-ALP-aniso-iso-CGO-relation} to derive a boundary integral equation (BIE) for the scattering transform $\tau(k)$ in terms of the boundary traces of the anisotropic \textsc{CGO} solutions, thus providing the crucial connection between our anisotropic data in the \emph{physical space}, and the isotropic version in the \emph{virtual space} (isothermal coordinates).

Let $W^\pm_\gamma(\zeta,k) =f_{\pm\mu_\gamma}(\zeta,k)$, for $k,\;\zeta\in\C$.  By Lemma~\ref{lem-ALP-aniso-iso-CGO-relation} part (iii) we have, in the exterior domain, 
\[W^\pm_\sigma(z,k)=W^\pm_\gamma(F(z),k) =f_{\pm\mu_\gamma}(F(z),k), \quad z\in\C\setminus\D.\]
As $F$ is a quasiconformal homeomorphism and diffeomorphism with asymptotics given in \eqref{eq-asympt-withA}, for $|F(z)|>R$ we can use the power series representation for $f_{\pm\mu_\gamma}(z,k)$ from \eqref{eq-f-expansion}  
\begin{eqnarray}
W^\pm_\sigma(z,k) &=&f_{\pm\mu_\gamma}(F(z),k)\nonumber\\
&=&e^{ikF(z)}\left[1+\frac{b^\pm_1(k)}{F(z)}+\frac{b^\pm_2(k)}{\left(F(z)\right)^2}+\cdots \right].\label{eq-W-sig-gam-expansion}
\end{eqnarray}
Using the first-order development of $F(z)$ in Lemma~\ref{lem: 1}, from \eqref{eq-asympt-withA}
\[F(z)=z+\frac{A}{z}+\O\left(\frac{1}{|z|^2}\right),\]
we see  
\begin{eqnarray}
W^\pm_\sigma(z,k) 
&=& e^{ikF(z)}\left[1+\frac{b^\pm_1(k)}{F(z)}+\frac{b^\pm_2(k)}{\left(F(z)\right)^2}+\cdots \right]\nonumber\\
&=& e^{ikz}\left[1+\frac{ikA+b^\pm_1(k)}{z}+\O_k\left(\frac{1}{|z|^2}\right)\right]\nonumber\\
&=& e^{ikz}\left[1+\frac{\tilde{b}^\pm_1(k)}{z}+\O_k\left(\frac{1}{|z|^2}\right)\right],\label{eq-W-sig-gam-expansion-b1tilde}
\end{eqnarray}
where 
\begin{equation}\label{eq-b1tilde}
\tilde{b}^\pm_1(k)=ikA+b^\pm_1(k).
\end{equation}
We can express the scattering transform $\tau(k)$ in terms of the new coefficient $\tilde{b}^\pm_1(k)$ as follows:
\begin{equation}\label{eq-scat-b1tilde}
\tau(k)= \frac{1}{2}\left(\overline{b^+_1(k)}-\overline{b^-_1(k)}\right)
= \frac{1}{2}\left(\overline{\tilde{b}^+_1(k)}-\overline{\tilde{b}^-_1(k)}\right),
\end{equation}
as the constant $A$ cancels. Thus we have related the isotropic Beltrami scattering transform to the anisotropic \textsc{CGO}s.

Next, we will derive boundary integral equations for the coefficients $\tilde{b}^\pm_1(k)$, and thus the isotropic Beltrami scattering transform $\tau(k)$, in terms of the boundary traces of the anisotropic Beltrami \textsc{CGO}s $M^\pm_\sigma$.  Recall that $M^\pm_\sigma(z,k)$ are analytic for $|z|\geq 1$, and from \eqref{eq-W-sig-gam-expansion-b1tilde} we have
\begin{eqnarray}
{M}^\pm_\sigma(z,k)
&=& e^{-ikz}W^\pm_\sigma(z,k)\nonumber\\
&=& e^{-ikz}f_{\pm\mu_\gamma}(F(z),k)\nonumber\\
&=&1+\frac{\tilde{b}^\pm_1(k)}{z}+\O_k\left(\frac{1}{|z|^2}\right), \quad |z|\to\infty\label{eq-Mtilde}.
\end{eqnarray}
Now, let
\begin{eqnarray}
h(z,k) 
&=& \frac{1}{z}\;{M}^\pm_\sigma\left(\frac{1}{z},k\right) -\frac{1}{z}\nonumber\\
&=& \frac{1}{z}\left[1+\tilde{b}^\pm_1(k)z+\O_k\left(|z|^2\right)\right] -\frac{1}{z}\nonumber\\
&=& \tilde{b}^\pm_1(k)+\O_k\left(|z|\right), \quad |z|\to 0\label{eq-h-func}.
\end{eqnarray}

By definition, $h(z,k)$ is analytic in $z$ for $z\in\D\setminus\{0\}$.  In fact, as $h(z,k)$ is bounded, we know that $h(z,k)$ is analytic at $z=0$ as well and thus by setting $h(0,k)= \tilde{b}^\pm_1(k)$, we extend $h(z,k)$ to all of $\D$.  By the Cauchy Integral Formula, the value of $h(z,k)$ for any $|z|<1$ is completely determined by the values at the boundary, and thus
\begin{eqnarray}
\tilde{b}^\pm_1(k)
&=& h(0,k)\nonumber\\
&=& \frac{1}{2\pi i}\int_{\partial\D} \frac{h(z,k)}{z-0}\;dz\nonumber\\
&=& \frac{1}{2\pi i}\int_{\partial\D} \frac{1}{z-0}\left[\frac{1}{z}\;{M}^\pm_\sigma\left(\frac{1}{z},k\right) -\frac{1}{z}\right]\;dz\nonumber\\
&=& \frac{1}{2\pi i}\int_{\partial\D} \frac{1}{z^2}\left[{M}^\pm_\sigma\left(\frac{1}{z},k\right) -1\right]\;dz\nonumber,
\end{eqnarray}
or equivalently, 
\begin{eqnarray}
\tilde{b}^\pm_1(k)
&=& \frac{1}{2\pi i}\int_{\partial\D} \left(M^\pm_\sigma(z,k) - 1\right)\;dz\label{eq-btilde-BIE-M-sigma}.
\end{eqnarray}
Thus, we can determine the scattering transform $\tau(k)$ using the traces of the \textsc{CGO} solutions $M^\pm_\sigma(z,k)$ to the anisotropic problem, in the physical coordinates, resulting in the following theorem which holds for anisotropic $\sigma$, and $\widehat{\sigma} \in L^\infty$.

\begin{theorem}\label{thm-scat}
The Beltrami scattering transform data $\tau(k)$ for the isotropic conductivity $\gamma=F_*\sigma$, where $F$ is as in Lemma~\ref{lem: 1}, can be calculated in terms of the boundary values of the anisotropic $\sigma$ and $\widehat{\sigma}$ Beltrami \textsc{CGO} solutions via
\[\tau(k)=\frac{\overline{\widetilde{b}^+_1(k)}-\overline{\widetilde{b}^+_1(k)}}{2},\]
where 
\[\widetilde{b}^\pm_1(k)=\frac{1}{2\pi i}\int_{\p\D}\left[M^\pm_\sigma(z,k)-1\right]\; dz.\]
\end{theorem}
\begin{proof}
The proof follows directly from \eqref{eq-scat-b1tilde} and \eqref{eq-btilde-BIE-M-sigma}.
\end{proof}

\subsubsection{From Scattering Data to Isotropic Conductivity $\gamma=\sqrt{\det\sigma}$}\label{sec-scat-to-Gam}
As the scattering transform $\tau(k)$ is independent of the spatial coordinates, we can now proceed with the established theory for the isotropic conductivity problem.  Following along the $L^\infty$ approach of \cite{Astala2006a} would require knowledge of the isotropic \textsc{CGO} solutions $M^\pm_\gamma$ in the exterior domain $\C\setminus\tilde{\Omega}$ and thus the explicit formation of the unstable map $F$ at the boundary $\p\widetilde{\Omega}=F(\bndry)$ via \cite{Astala2005}  
\begin{equation}\label{eq-Fz-lim}
F(z)=\underset{k\to\infty}{\lim}\frac{\log W^+_\sigma(z,k)}{ik},\quad z\in\C\setminus\Omega.
\end{equation}
Note that the logarithm is not necessarily the principal branch.  For any $k\not=0$ we consider the logarithm of  $G(z,k)=W^+_\sigma(F(z),k)$ where the  branch is chosen so that $\log G(z,k)$ is a continuous function of $z\in\C\setminus\Omega$ and $\lim_{z\to \infty} (\log G(z,k)-ikz)=0$. Then one has $\lim_{k\to \infty} \frac{\log G(z,k)}{ik}=F(z)$.  Note that such a choice of the branch of the logarithm is not numerically feasible as one would need to compute values of $G(z,k)$ for both large $z$ and $k$, when we only have the boundary traces for small magnitude $k$. 

Instead, we bypass constructing the map $F$ completely by relating the $L^\infty$ isotropic Beltrami scattering transform $\tau(k)$ to the $C^2$ isotropic Schr\"odinger scattering transform $\mathbf{t}(k)$ via \cite{Astala2010}
\begin{equation}\label{eq-relate-AP-Nach-scattering}
\mathbf{t}(k)=-4\pi i\bar{k} \tau(k),
\end{equation}
and continue by solving the $C^2$ isotropic conductivity problem \cite{Nachman1996}.  Therefore, for each fixed $\zeta\in\widetilde{\Omega}$, we solve the following $\dbark$ equation 
\begin{equation}\label{eq-dbark-C2-1st} 
\dbark \mathcal{M}(\zeta,k)= \frac{1}{4\pi\bar{k}}\mathbf{t}(k) e(\zeta,-k)\overline{\mathcal{M}(\zeta,k)},
\end{equation}
for $\mathcal{M}(\zeta,k)$, $k\in\C$, where $e(\zeta,k):=\exp(i(k\zeta+\overline{k}\overline{\zeta}))$ is a unimodular multiplier.  Here  $\mathcal{M}(\zeta,k)$ are the \textsc{CGO} solutions of Nachman \cite{Nachman1996}, and  $\mathcal{M}(\zeta,k)-1 \in W^{1,\tilde{p}}(\widetilde{\Omega})$ for $\tilde{p}>2$.  The isotropic conductivity $\gamma$ is then recovered from the low frequency \textsc{CGO}s via 
\begin{equation}\label{eq-M-to-GAM}
\gamma(\zeta)=\mathcal{M}(\zeta,0)^2.
\end{equation}

As the $C^2$ isotropic case is well studied in the theoretical and numerical settings, we see that once we obtain the scattering data $\tau(k)$, we can recover the isotropic representation $\gamma(\zeta)=(F_* \sigma)(\zeta)$ of the anisotropic conductivity $\sigma$.  What remains to be shown now is a way to determine the anisotropic traces of $M^\pm_\sigma$ for $z\in\p\D$, which we will now do. 

\subsection{Determining the traces of the anisotropic \textsc{CGO} Solutions}\label{sec-CGOtraces}
Recall, from Section~\ref{eq-scat-BIEs}, that for fixed $k\in\C$, $W^+_\sigma(z,k)$ denotes the restriction of the anisotropic $\sigma$  {\sc CGO} solution $g(z,k)$ of \eqref{eq-aniso-Beltrami} to the exterior domain $\C\setminus\D$, and satisfies \eqref{eq-dbarz Wplus-sig}-\eqref{eq-Wplus-sigma-BIE} 
\begin{eqnarray*}
\dbarz W^+_\sigma(z,k) &=& 0, \quad \text{for $z\in\C\setminus\D$ }\\
W^+_\sigma(z,k)&=& e^{ikz}\left(1+\O_k\left(\frac{1}{|z|}\right)\right)\\
\im \left.W^+_\sigma(z,k)\right|_{z\in\partial\D}&=& \HT_{\sigma}\left(\re \left.W^+_\sigma(z,k)\right|_{z\in\partial\D}
\right).
\end{eqnarray*}
Similarly, for fixed $k\in\C$, $W^-_\sigma(z,k)$ denotes the restriction of the anisotropic $\widehat{\sigma}$ {\sc CGO} solution $\tilde{g}(z,k)$ of \eqref{eq-aniso-Beltrami-Minusmu2} to the exterior domain $z\in\C\setminus\D$, and satisfies \eqref{eq-dbarz Wminus-sig}-\eqref{eq-Wminus-sigma-BIE}
\begin{eqnarray*}
\dbarz W^-_\sigma(z,k) &=& 0, \quad \text{for $z\in\C\setminus\D$ }\\
W^-_\sigma(z,k)&=& e^{ikz}\left(1+\O_k\left(\frac{1}{|z|}\right)\right)\\
\im \left.W^-_\sigma(z,k)\right|_{z\in\partial\D}&=& \HT_{\widehat{\sigma}}\left(\re \left.W^-_{{\sigma}}(z,k)\right|_{z\in\partial\D}
\right).
\end{eqnarray*}
We will use the boundary relations involving the $\sigma$ and $\widehat{\sigma}$-Hilbert transforms $\H_\sigma$ and $\H_{\widehat{\sigma}}$ to derive the boundary integral equations for $M^\pm_\sigma(z,k)=e^{-ikz}W^\pm_\sigma(z,k)$ for $z\in\p\D$.  

Let us first consider the $W^+_\sigma$ case.  Write $g=W^+_\sigma=v+iw$ where $v=\re\left(W^+_\sigma\right)$ and $w=\im\left(W^+_\sigma\right)$.  Then, 
\[\nabla\cdot\sigma\nabla v(z,k)=0,\qquad \nabla\cdot\widehat{\sigma}\nabla w(z,k)=0.\]
By the definition of $\H_\sigma$ in \eqref{eq-Hmu-def-aniso}, $\H_\sigma\left(v|_{\bndry}\right)=w|_{\bndry}$.

Currently the map $\H_\sigma$ is only defined for real-valued functions.  To determine its action on a purely imaginary function now consider $h=-ig=w-iv$.  As $g$ satisfies $\eqref{eq-aniso-Beltrami}$ we see that $h$ satisfies 
\[\dbarz h(z,k)=\mu_1(z)\dez h(z,k)-\mu_2(z)\overline{\dez h(z,k)},\quad z,k\in\C,\]
which turns out to be the Beltrami equation corresponding to the anisotropic conductivity $\widehat{\sigma}$.  Thus $\H_{\widehat{\sigma}}(w)=-v$.  As $w=\H_\sigma(v)=\re(h)$ it is natural to set
\begin{equation}\label{eq-H-sig-sigHat-neg-relation}
\H_{\widehat{\sigma}}(w)=i\H_\sigma(iw),
\end{equation} 
therefore extending the definition of $\H_\sigma$ to all complex-valued functions $g\in H^{1/2}(\bndry)$.  Moreover, we have 
\begin{equation}\label{eq-Hsig-HsigMat-composition}
{\H_\sigma\circ\H_{\widehat{\sigma}}(g)=\H_{\widehat{\sigma}}\circ\H_{{\sigma}}(g)
=-g+\frac{1}{2\pi}\int_{\p\D}g\;dS.}
\end{equation}
For more details (in terms of the isotropic problem), see \cite[Section 16.3]{Mueller2012}.

If $\sigma=I_{2\by2}$, then $\mu_1=0=\mu_2$ and the problem reduces to the isotropic $\gamma=1$ (and $\mu=0$) case.  We briefly review the setting here. The map resulting from $\sigma=I_{2\by2}$ or equivalently $\gamma=1$ is $\H_0$, the standard Hilbert transform on the unit circle.  $\H_0$ is a singular operator with Fourier multiplier $m(\xi)=-i\xi/|\xi|$ defined for $\xi\in\Z\setminus{0}$ with $m(0)=0$ such that
\[\left\{\H_0 g(\cdot)\right\}^\wedge(\xi)=m(\xi)\left\{g(\cdot)\right\}^\wedge(\xi)\quad \text{for $g\in L^2(\p\D)$}.\]
As in \cite{Astala2006}, we will make use of the projection $\mathcal{P}_0$ of $\H_0$ in the boundary integral equations for the \textsc{CGO}s.  The Riesz projection  $\mathcal{P}_0$ onto the Hardy spaces on $\p\D$ is determined by $\H_0$ via
\begin{equation}\label{eq-P0}
\mathcal{P}_0g=\frac{1}{2}\left(I+i\H_0\right)g+\frac{1}{2}\Ave g, 
\end{equation}
where $\Ave$ is an averaging operator 
\begin{equation}\label{eq-ave-operator}
  \Ave\phi:=|\partial\Om|^{-1}\int_{\partial\Om}\phi \,dS.
\end{equation}
In the same vain as the isotropic case, we define the projection $\mathcal{P}_\sigma$ 
\begin{eqnarray}
\mathcal{P}_\sigma g=\frac{1}{2}\left(I+i\H_\sigma\right)g+\frac{1}{2}\Ave g,\quad \text{$g\in H^{1/2}(\p\D)$.} \label{eq-Psig}
\end{eqnarray}
Conjugating with the exponential function yields the $k$ dependent operator defined by 
\begin{eqnarray}
\mathcal{P}^k_{{\sigma}}(g)(z)=e^{-ikz}\mathcal{P}_{{\sigma}}\left(e^{ik\cdot}g\right)(z),\quad \text{$g\in H^{1/2}(\p\D)$.}\label{eq-Pk-sig}
\end{eqnarray}

\begin{theorem}\label{thm-CGOtraces}
For each fixed $k\in\C$, the boundary traces of the \textsc{CGO} solutions $M^\pm_\sigma(z,k)$ on $\p\D$ for the anisotropic $\sigma$ and $\widehat{\sigma}$ Beltrami problems \eqref{eq-aniso-Beltrami} and \eqref{eq-aniso-Beltrami-Minusmu2} are the unique solutions to
\begin{eqnarray}
M^+_\sigma(z,k)+1&=&\left(\mathcal{P}_0+\mathcal{P}^k_\sigma\right) M^+_\sigma(z,k),\quad z\in\p\D\label{eq-Mplus-sig-BIE}\\
M^-_\sigma(z,k)+1&=&\left(\mathcal{P}_0+\mathcal{P}^k_{\widehat{\sigma}}\right) M^-_\sigma(z,k),\quad z\in\p\D\label{eq-Mminus-sig-BIE}.
\end{eqnarray}
\end{theorem}
\begin{proof}
The $\sigma$ and $\widehat{\sigma}$ proofs are analogous to the isotropic $\gamma$ and $1/\gamma$ cases with the underlying isotropic Beltrami equations \eqref{eq-u-gamma-beltrami} replaced with the their corresponding anisotropic versions \eqref{eq-aniso-Beltrami} and \eqref{eq-aniso-Beltrami-Minusmu2}, and the Hilbert transform  $\H_\sigma$ extended to all complex-valued functions $g\in H^{1/2}(\bndry)$ as in 
\eqref{eq-H-sig-sigHat-neg-relation}.  See \cite{Astala2006a} and \cite[Section 16.3]{Mueller2012} for more details. 
\end{proof}

\section{A Direct Nonlinear D-bar Algorithm}\label{sec-algorithm}
The constructive proof described above in Section~\ref{sec-CGOproof} corresponds to a direct nonlinear D-bar algorithm which we summarize here.  The reconstruction method from infinite-precision data consists of the following steps:
\[\Lambda_\sigma \overset{1}{\longrightarrow} \left.M^\pm_\sigma(z,k) \right|_{\bndry} \overset{2}{\longrightarrow} \tilde{b}_1^\pm(k)\overset{3}{\longrightarrow}\mathbf{t}(k)
\overset{4}{\longrightarrow}\gamma.\]
\begin{itemize}
\item[{\bf Step 1:}] {\bf From boundary measurements $\Lambda_\sigma$ to traces of the anisotropic \textsc{CGO} solutions $M^\pm_\sigma$:}\\ For each fixed $k\in\C$, solve the following boundary integral equations for $M^\pm_\sigma$
\begin{eqnarray}
M^+_\sigma(\cdot,k)|_{\bndry} +1 &=& \left(\mathcal{P}^k_{\sigma}+\mathcal{P}_0\right)M^+_\sigma(\cdot,k)|_{\bndry} \label{eq-M-aniso-traces-Mplus}\\
M^-_\sigma(\cdot,k)|_{\bndry} +1 &=& \left(\mathcal{P}^k_{\widehat{\sigma}}+\mathcal{P}_0\right)M^-_\sigma(\cdot,k)|_{\bndry} \label{eq-M-aniso-traces-Mminus}
\end{eqnarray}
where $\mathcal{P}^k_{\sigma}$,  $\mathcal{P}^k_{\widehat{\sigma}}$, and $\mathcal{P}_0$ are the projection operators explained in detail in Section~\ref{sec-CGOtraces}.

\vspace{1em}
\item[{\bf Step 2:}] {\bf From the anisotropic \textsc{CGO} traces $M^\pm_\sigma$ to the isotropic Beltrami scattering data $\tilde{b_1^{\pm}}(k)$:} \\ Substitute the anisotropic traces of 
$M^\pm_\sigma$ into the formula for the nonlinear isotropic Beltrami scattering data
\begin{equation}\label{eq-btilde-intEq}
\tilde{b_1^{\pm}}(k) = \frac{1}{2\pi i}\int_{\bndry}\left[M^\pm_\sigma(z,k) -1\right]\;dz.
\end{equation}

\vspace{1em}
\item[{\bf Step 3:}]  {\bf From Beltrami Scattering data $\tilde{b_1^{\pm}}(k)$ to Schr\"{o}dinger Scattering data $\mathbf{t}(k)$:}\\
For each $k$, evaluate 
\begin{equation}\label{eq-t-def}
\mathbf{t}(k) = -4\pi i \bar{k}\,\tau(k),
\end{equation}
where
\begin{equation}\label{eq-tau-def}
\tau(k) = \frac{\overline{\tilde{b}_1^+(k)} - \overline{\tilde{b}_1^-(k)}}{2}.
\end{equation}

\vspace{1em}
\item[\textbf{Step 4:}] \textbf{From the isotropic scattering data to the isotropic conductivity:} \\ For each fixed $\zeta\in \tilde{\Om}$, solve the $\dbark$-equation
    \begin{align}\label{eq-dbark-C2}
     \dbark \mathcal{M}(\zeta,k)= \frac{1}{4\pi\bar{k}}\mathbf{t}(k) e(\zeta,-k)\overline{\mathcal{M}(\zeta,k)},
    \end{align}
where $\mathcal{M}\sim1$ for $|\zeta|\to\infty$.
  The $C^2$ representative isotropic conductivity is then recovered by $\gamma(\zeta)=\mathcal{M}(\zeta,0)^2$ up to a change of coordinates. Note that as $\gamma$ is the isotropic representation, we have $\gamma(\zeta) =\left(F_*\sigma\right)(\zeta)= \sqrt{\det\left(\sigma\left(F^{-1}(\zeta)\right)\right)}$.
\end{itemize}

\subsection{Numerical Setup}\label{sec-algorithm-setup}
In Step~4 of the algorithm, we need to solve the $\dbark$ equation 
\[\dbark \mathcal{M}(\zeta,k)= \frac{1}{4\pi\bar{k}}\mathbf{t}(k) e(\zeta,-k)\overline{\mathcal{M}(\zeta,k)},\]
from Nachman's constructive proof for $C^2$ isotropic conductivities \cite{Nachman1996}.   To solve the $\dbark$ equation numerically, we need to truncate the scattering data.  In \cite{Knudsen2009}, it was shown that truncation of the scattering data $\mathbf{t}(k)$ corresponds to a regularization strategy.  In this spirit, fix $R>0$, a positive radius for the $k$-parameter, and define the truncated scattering transform
\begin{equation}\label{eq-dbark-C2-truncated}
 \mathbf{t}^R(k)=\begin{cases}
 \mathbf{t}(k) & |k|\leq R\\
 0 & |k|>R.
 \end{cases}
\end{equation}
To solve the $\dbark$ equation, we use the numerical solution method of Knudsen, Mueller and Siltanen~\cite{Knudsen2004a}, which uses Fourier transforms, and thus we work on the same special $k$-grid, with $k\in[-R,R]^2$, defined as follows.  

Choose $N_k=2^c$, the number of discretization points in the $k_i$ direction ($i=1,2$), for some positive integer $c$.  Set the stepsize in $k$ to be $h_k=\frac{2R}{N_k}$, and form the computational $k$-grid $\mathcal{G}_c$ by
\begin{equation}\label{eq-kGrid-Gc}
\mathcal{G}_c=\left\{\mathbf{j} h_k\;\middle|\; \mathbf{j}\in\Z^2_c\right\},
\end{equation}
where 
\begin{equation}\label{eq-kGrid-Z2c}
\Z^2_c=\left\{\mathbf{j}=\left(j_1,j_2\right)\in\Z\by\Z\;\middle|\; -2^{c-1}\leq j_i<2^{c-1} \quad i=1,2\right\}.
\end{equation}
Note that there are $N_k^2=2^{2c}$ points in the computational $k$-grid $\mathcal{G}_c$.  Furthermore, the explicit construction in \eqref{eq-kGrid-Gc}-\eqref{eq-kGrid-Z2c} (which excludes the $j_i=2^{c-1}$ entries for $i=1,2$) is essential for the periodic solution technique used in the solution of the $\dbark$ equation.

We now will describe the numerical implementation of each step of the algorithm.
\subsection{Step 1: Solving the Anisotropic Boundary Integral Equations}\label{sec-algorithm-BIE}
The numerical solution of the boundary integral equations \eqref{eq-M-aniso-traces-Mplus} and \eqref{eq-M-aniso-traces-Mminus} is done by writing the real and imaginary parts separately, replacing all the operators by matrix approximations, and solving the resulting finite linear system.  As the boundary integral equations are nearly identical to their isotropic counterparts \cite{Astala2006} they can be solved in the same manner as \cite[Section 16.3]{Mueller2012}.  We review the process here for $M^+_\sigma$.  The process is analogous for $M^-_\sigma$.

We will consider the trigonometric basis functions 
\begin{equation} \label{Trig}
\phi_n(\theta)= 
\left \{ \begin{array}{lc}
\displaystyle \pi^{-1/2}\cos\left((n+1)\theta/2\right), \quad \mbox{ for odd }n, \\
\displaystyle \pi^{-1/2}\sin\left(n\theta/2\right), \quad \mbox{ for even } n.
\end{array}\right. 
\end{equation}
The Dirichlet-to-Neumann map $\Lambda_\sigma$ is approximately represented by the matrix $\mathbf{L}_\sigma$ defined by
\begin{equation}\label{eq-DN-mat}
(\mathbf{L}_\sigma)_{m,n} := \langle \Lambda_\sigma\phi_n,\phi_m\rangle=\int_{\bndry}\left(\Lambda_\sigma\phi_n\right)\overline{\phi_m}.
\end{equation}
The tangential derivative map\index{tangential derivative map} $\dd_T$ can be approximated in the basis \eqref{Trig} by the  matrix $\mathbf{D}_T$:
\begin{equation}\label{dermatrix}
  \mathbf{D}_T =  \left[\begin{array}{rrrrrrrr}
  {0} &   {1}  &  \phantom{-} &  \phantom{-}&    &      &        & \\
   {-1}  & {0} &  \, & \,  &    &      &        & \\
    \,    & \,    &{0}  & {2}  &    &      &        & \\
     \,   &  \,   &  {-2} & {0}  &    &      &        & \\
              &     &   &   &    & \ddots &        & \\
              &     &   &   &    &      & {0} & {N}\\
              &     &   &   &    &      &    {-N}    & {0}
      \end{array}
  \right].
\end{equation}
We approximate $\Hil_\sigma$ acting on real-valued, zero-mean functions expanded in the basis \eqref{Trig} by 
\begin{equation}\label{matriisihassakka}
  \widetilde{\mathbf{H}}_{\sigma} := \mathbf{D}_T^{-1}  \,\mathbf{L}_\sigma.
\end{equation}

In general, the traces of the \textsc{CGO} solutions at $\p\D$ do not have mean zero, and so we append the basis function $\phi_0=(2\pi)^{-1/2}$ to \eqref{Trig}.
This leads to the following $(2N+1)\times(2N+1)$ matrix approximation to the $\sigma$-Hilbert transform $\Hil_\sigma$:
\begin{equation}\label{H_mu_matrix}
  \mathbf{H}_{\sigma} :=\left[\!\begin{array}{lc}
  0_{1\by 1} & 0_{1\by 2N}\\
  0_{2N\by 1} & \widetilde{\mathbf{H}}_\sigma
  \end{array}\!\right].
\end{equation}
Furthermore, we can approximate $\Hil_{\widehat{\sigma}}$ when we have $\mathbf{H}_{\sigma}$ available. We have the identity
\begin{equation}\label{Hminusident}
  \Hil_{\widehat{\sigma}}\circ(-\Hil_{\widehat{\sigma}})u =   (-\Hil_{\widehat{\sigma}})\circ \Hil_\sigma u = u -\Ave u,
\end{equation}
so $-\Hil_{\widehat{\sigma}}$ is the inverse operator of $\Hil_\sigma$ in the subspace of zero-mean functions. Thus we may define a $(2N+1)\times(2N+1)$ matrix approximation to  $\Hil_{\widehat{\sigma}}$:
\begin{equation}\label{H_mmu_matrix}
  \mathbf{H}_{\widehat{\sigma}} :=\left[\!\begin{array}{lc}
  0_{1\by 1} & 0_{1\by 2N}\\
  0_{2N\by 1} & -\widetilde{\mathbf{H}}_\sigma^{-1}
  \end{array}\!\right].
\end{equation}

We represent complex-valued functions $g\in H^{1/2}(\DOm)$ by expanding the real and imaginary parts separately and organizing the transform coefficients as the following vertical vector in $\R^{4N+2}$ :
$$
 \widehat{g} = \left[\!
\begin{array}{c}
\langle \mbox{Re\,}g,\phi_0\rangle \\
\langle \mbox{Re\,}g,\phi_1\rangle\\
\vdots\\
\langle \mbox{Re\,}g,\phi_{2N}\rangle\\
\langle \mbox{Im\,}g,\phi_0\rangle\\
\langle \mbox{Im\,}g,\phi_1\rangle\\
\vdots\\
\langle \mbox{Im\,}g,\phi_{2N}\rangle
\end{array}\!\right]\in \R^{4N+2}.
$$
Denote the direct and inverse transform by $\widehat{g}=\widetilde{\mathcal{F}}g$ and $g=\widetilde{\mathcal{F}}^{-1}\widehat{g}$. Note that the $(4N+2)\times(4N+2)$ matrix representation of the averaging operator $\mathcal{L}$ is given by 
$$
  \mathbf{L}=\frac{1}{\sqrt{2\pi}}\mbox{diag}[1,0,\dots,0,1,0,\dots,0],
$$
where the ones are located at elements $(1,1)$ and $(2N+2,2N+2)$.

Now we can solve the boundary integral equation \eqref{eq-M-aniso-traces-Mplus} approximately by solving the following equation for the transform coefficients of $M^+_\sigma$:
\begin{equation}\label{AP_BIEmatrix}
  (I-\mathbf{P}_\sigma^k-\mathbf{P}_0)\widehat{M}^+_\sigma(\,\cdot\,,k)|_{\partial\D} = -\widetilde{\mathcal{F}}(1). 
\end{equation}
Here $\mathbf{P}_\sigma^k$ and $\mathbf{P}_0$ stand for approximate implementations of the actions of the operators ${\mathcal P}_\sigma^k$ and ${\mathcal P}_0$, respectively. The action of ${\mathcal P}_\sigma$ (defined in \eqref{eq-Psig}) in the transform domain is  
\[ \mathbf{P}_\sigma \,\widehat{g} = \frac{1}{2}\left(I + i\left[\!\begin{array}{cc}
  \mathbf{H}_{\sigma} & 0\\
  0 & \mathbf{H}_{\widehat{\sigma}} 
  \end{array}\!\right] \right) \widehat{g} +  \frac{1}{2}\mathbf{L} \,\widehat{g},\]
where $\mathbf{H}_{\sigma}$ and $\mathbf{H}_{\widehat{\sigma}}$ are given by \eqref{H_mu_matrix} and \eqref{H_mmu_matrix}, respectively. The action of ${\mathcal P}_\sigma^k$ (defined in \eqref{eq-Pk-sig}) in the transform domain is  
\begin{equation}\label{Pmu_kaction}
  \mathbf{P}_\sigma^k\,\widehat{g}= \widetilde{\mathcal{F}}\left(e^{-ikz} \cdot \widetilde{\mathcal{F}}^{-1}\left({\mathbf P}_\sigma\, \widetilde{\mathcal{F}}\big(e^{ikz}\cdot (\widetilde{\mathcal{F}}^{-1}\widehat{g})\big)\right)\right).
\end{equation}
Now one can solve equation \eqref{AP_BIEmatrix} for the transform coefficients $\widehat{M}^+_\sigma$ iteratively using {\sc GMRES}.

\subsection{Step 2: Computing the Beltrami Scattering Data}
For the given value $k\in\C$, we compute  $\tilde{b}_1^{\pm}(k)$ from \eqref{eq-btilde-intEq}
\[\tilde{b}_1^{\pm}(k) = \frac{1}{2\pi i}\int_{\bndry}\left[M^\pm_\sigma(z,k) -1\right]\;dz.\]
Letting $z=e^{i\theta}\in\p\D$, $\theta\in[0,2\pi)$ we have $dz=ie^{i\theta}d\theta$ and 
\[\tilde{b}_1^{\pm}(k) = \frac{1}{2\pi}\int_0^{2\pi}\left[M^\pm_\sigma(e^{i\theta},k) -1\right]e^{i\theta}\;d\theta.\]
We can now approximate the continuous integral above using a finite sum:
\begin{equation}\label{eq-b1-tilde-sum-approx}
\tilde{b_1^{\pm}}(k) \approx \frac{\Delta\theta}{2\pi}\sum_{n=1}^{N_z}\left[M^\pm_\sigma(e^{i\theta_n},k) -1\right]e^{i\theta_n},
\end{equation}
where $n=1,\ldots,N_z$ are the indices for the function values $M^\pm_\sigma(z_n,k)$ obtained numerically in Step~1 above, and $\Delta\theta$ denotes the stepsize in $\theta$ along the unit circle.
\subsection{Step 3: Computing the Schr\"{o}dinger Scattering Data}
The isotropic (truncated) Schr\"odinger scattering data $\mathbf{t}^R(k)$ is then computed by \eqref{eq-t-def}
\begin{equation}\label{eq-dbark-C2-truncated-tau}
 \mathbf{t}^R(k)=\begin{cases}
 -4\pi i\bar{k}\tau(k) & |k|\leq R\\
 0 & |k|>R,
 \end{cases}
\end{equation}
where $\tau(k)$ is evaluated via \eqref{eq-tau-def}
\[\tau(k) = \frac{\overline{\tilde{b}_1^+(k)} - \overline{\tilde{b}_1^-(k)}}{2}.\]
\subsection{Step 4: Solving the D-bar Equation}
As $\mathcal{M}\sim 1$ for $|\zeta|\to\infty$,  and $\frac{1}{\pi k}$ is the fundamental solution for the $\dbark$ operator, the solution to the  $\dbark$ equation
\[\dbark \mathcal{M}(\zeta,k)=\frac{1}{4\pi\bar{k}} \mathbf{t}(k) e(\zeta,-k)\overline{\mathcal{M}(\zeta,k)},\]
can be written as 
\begin{equation}\label{eq-dbark-C2-sol}
\mathcal{M}(\zeta,k)=1+\frac{1}{\pi}\int_{\R^2} \frac{\mathbf{t}(\kappa) e(\zeta,-\kappa)\overline{\mathcal{M}(\zeta,\kappa)}}{4\pi\bar{\kappa}(k-\kappa)}\;d\kappa_1 d\kappa_2,
\end{equation}
or in terms of convolutions
\begin{equation}\label{eq-dbark-C2-sol-conv}
\mathcal{M}(\zeta,k)=1+\frac{1}{\pi k}\ast\left( \frac{1}{4\pi\bar{k}} \mathbf{t}(k) e(\zeta,-k)\overline{\mathcal{M}(\zeta,k)}\right),
\end{equation}
where the convolution takes place in the $k$ variable.  Note that the integration takes place over the entire plane.  Using the truncated scattering data $\mathbf{t}^R$ we can instead find the regularized solution $\mathcal{M}_R(\zeta,k)$ by solving
\begin{equation}\label{eq-dbark-C2-sol-trunc}
\mathcal{M}_R(\zeta,k)=1+\frac{1}{\pi}\int_{|\kappa|\leq R} \frac{\mathbf{t}^R(\kappa) e(\zeta,-\kappa)\overline{\mathcal{M}_R(\zeta,\kappa)}}{4\pi\bar{\kappa}(k-\kappa)}\;d\kappa_1 d\kappa_2,
\end{equation}
or in terms of convolutions
\begin{equation}\label{eq-dbark-C2-sol-conv-trunc}
\mathcal{M}_R(\zeta,k)=1+\frac{1}{\pi k}\ast\left( \frac{1}{4\pi\bar{k}} \mathbf{t}^R(k) e(\zeta,-k)\overline{\mathcal{M}_R(\zeta,k)}\right),
\end{equation}
where the integration now takes place over a disc of radius $R$.

From \cite{Mueller2003, Knudsen2009} this gives the correct result as $R\to\infty$.  To solve for $\mathcal{M}_R$ numerically, we exploit the convolution using Fourier transforms and Vainikko's fast solution method on the special $k$-grid defined above in Section~\ref{sec-algorithm-setup} by
\begin{eqnarray}
\mathcal{M}_R(\zeta,k) 
&=& 1 + \mathcal{F}^{-1}\left\{\mathcal{F}\left(\frac{1}{\pi k}\ast \frac{1}{4\pi\bar{k}} \mathbf{t}^R(k) e(\zeta,-k)\overline{\mathcal{M}_R(\zeta,k)}\right)\right\}\nonumber\\
&=& 1 + \mathcal{F}^{-1}\left\{\mathcal{F}\left(\frac{1}{\pi k}\right) \mathcal{F}\left(\frac{1}{4\pi\bar{k}} \mathbf{t}^R(k) e(\zeta,-k)\overline{\mathcal{M}_R(\zeta,k)}\right)\right\}\label{eq-dbark-C2-sol-trunc-FFT}.
\end{eqnarray}
The singularity at $k=0$ is dealt with in the following way.  The $k=0$ entry for ${\mathbf{t}^R(k)}/{\bar{k}}$ is set to zero, its analytic limiting value.  Additionally, the fundamental solution $\frac{1}{\pi k}$ is set to zero for $k=0$ (see \cite{Knudsen2004} for further details).  The resulting system in \eqref{eq-dbark-C2-sol-trunc-FFT} can then be solved using matrix-free GMRES.

After solving \eqref{eq-dbark-C2-sol-trunc-FFT} for a given $\zeta\in\widetilde{\Omega}$, the $k=0$ entry $\mathcal{M}_R(\zeta,0)$ is saved.  This step is repeated (in parallel if desired) for each $\zeta\in\tilde{\Omega}$.  The isotropic (regularized) $C^2$ conductivity $\gamma_R$ is then determined by
\begin{equation}\label{eq-M-to-gam-regularized}
\gamma_R(\zeta)=\left(\mathcal{M}_R(\zeta,0)\right)^2\approx\sqrt{\det\left(\sigma(F^{-1}(\zeta)\right)}.
\end{equation}

\section{Numerical Results and Discussion}\label{sec-results}
We tested the algorithm presented in Section~\ref{sec-algorithm} on simulated anisotropic EIT data for both  $C^2$ smooth conductivities and piecewise-constant conductivities (see Figures~\ref{fig-phantoms-PIPES} and \ref{fig-phantoms-HnL}).  For each phantom, we solved the anisotropic conductivity equation \eqref{eq-conduct}
\[\nabla\cdot\sigma\nabla u=0,\]
with Neumann boundary condition $\sigma\frac{\p u}{\p\nu}=\phi_j$ on $\bndry$ defined by \eqref{Trig}, for $j=1,\ldots 32$, using the Finite Element Method (FEM).  The numerical FEM solution was used to form a $33\by33$ discrete matrix approximation $\mathbf{R}_\sigma$ to the Neumann-to-Dirichlet map, and subsequently the discrete matrix approximation $\mathbf{L}_\sigma$ to the D-N map $\Lambda_\sigma$ as described in \cite[Section 13.2]{Mueller2012}.

The traces of the anisotropic \textsc{CGO} solutions $M^\pm_\sigma$ were solved numerically (via \eqref{AP_BIEmatrix}), for $z=e^{i\theta}\in\bndry$ where
\begin{equation*}\label{eq-theta-BIE-values}
\theta\in\frac{2\pi}{33}\left[-16,-15,\ldots,-1,0,1,\ldots, 15, 16\right].
\end{equation*}
The scattering radius $R$ and corresponding $k$-grid are specific to each problem and are stated below.  In each example, the $\dbar_k$ equation \eqref{eq-dbark-C2} was solved for $|\zeta|\leq 1.2$ with step-size $h_\zeta=0.0094$, where the truncated isotropic representative conductivity $\gamma_R(\zeta)\approx(F_*\sigma)(\zeta)=\sqrt{\det\sigma\left(F^{-1}(\zeta)\right)}$ was recovered.

Our FEM computations were performed using {\sc Matlab}'s PDE toolbox. The basis functions were piecewise linear, and the triangular mesh generated for the unit disc domain comprised of 262,144 triangles and 131,585 vertices. The conductivity values were specified at the barycenter of each triangle. Thus, strictly speaking, none of our computational conductivity models were really $C^2$ but rather piecewise constant in a relatively fine triangular tiling, and the conductivity function used to specify the values of the coefficient $\sigma$ was $C^2$ smooth.  A more detailed analysis of the effect of piecewise constant approximations is outside the scope of this paper.

Without parallelization, the algorithm takes approximately 5 seconds on a standard laptop when using a $32\by 32$ scattering $k$-grid and reconstructing on a $32\by32$ spatial $\zeta$-grid.  The reconstructions shown here were computed using a $128\by 128$ scattering grid and $256\by 256$ spatial grid, to show higher resolution, and took approximately 45 minutes each (again, without parallelization).  We note that the algorithm is parallelizable in the following steps: $\Lambda_\sigma\longrightarrow \mathbf{t}(k)$ (in $k$), and $\mathbf{t}(k)\longrightarrow \gamma_R(\zeta)$ (in $\zeta$).  A detailed study of optimizing the algorithm for faster reconstructions is outside the scope of this paper.

\subsection{Test~1: Two Circular Inclusions}
Our first test involves the $C^2$ smooth anisotropic phantom shown in Figure~\ref{fig-phantoms-PIPES} with two circular inclusions with orthogonal directional preferences
\begin{equation}\label{eq-sig-PIPES-sig1-sig2}
\sigma_1=\left[\begin{array}{cc}
1 & 0\\
0 &4\\
\end{array}\right],\qquad
\sigma_2=\left[\begin{array}{cc}
2 & 0\\
0 &1\\
\end{array}\right].
\end{equation}
The right inclusion is more conductive in the vertical direction whereas the left inclusion more conductive in the horizontal direction. 
\begin{figure}[h!] 
\centering
\begin{picture}(140,130)
\put(0,-5){\includegraphics[height=125pt]{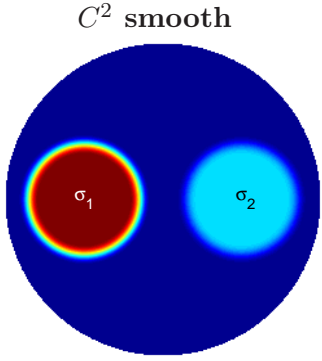}}
\put(30,122){\textbf{$C^2$ smooth}}
\end{picture}
\caption{\label{fig-phantoms-PIPES}The $C^2$ smooth two circular inclusions phantom used in Test~1. The conductivities in the left circular region, $\sigma_1$,  and in the right circular region $\sigma_2$ are defined in \eqref{eq-sig-PIPES-sig1-sig2}.  The background conductivity is the $2\by 2$ identity matrix.}
\end{figure}

The boundary integral equation \eqref{AP_BIEmatrix} was solved to recover $M^\pm_\sigma$ for $z\in\p\D$ as described above, with $|k|\leq 6$ to evaluate the scattering transform $\mathbf{t}_R(k)$ on a $k$-grid with a step-size of $h_k\approx 0.094$.  The $\dbark$-equation was then solved, and the truncated isotropic representative conductivity $\gamma_R(\zeta)\approx(F_*\sigma)(\zeta)=\sqrt{\det\sigma\left(F^{-1}(\zeta)\right)}$ was recovered.
Figure~\ref{fig-recon-PIPES-c2} shows the reconstructed isotropization $\gamma_R(\zeta)$ (Right) and the true $\sqrt{\det\sigma(\zeta)}$ (Middle) in the deformed coordinates, as well as the true $\sqrt{\det\sigma(z)}$ in the physical coordinates.  In $\sqrt{\det\sigma}$, the prescribed maximum value of the left inclusion is 2, and 1.41 in the right inclusion.  The reconstructed maximums are 2.34 and 1.61, respectively.

\begin{figure}[h!] 
\centering
\begin{picture}(410,150)
\put(0,0){\includegraphics[height=130pt]{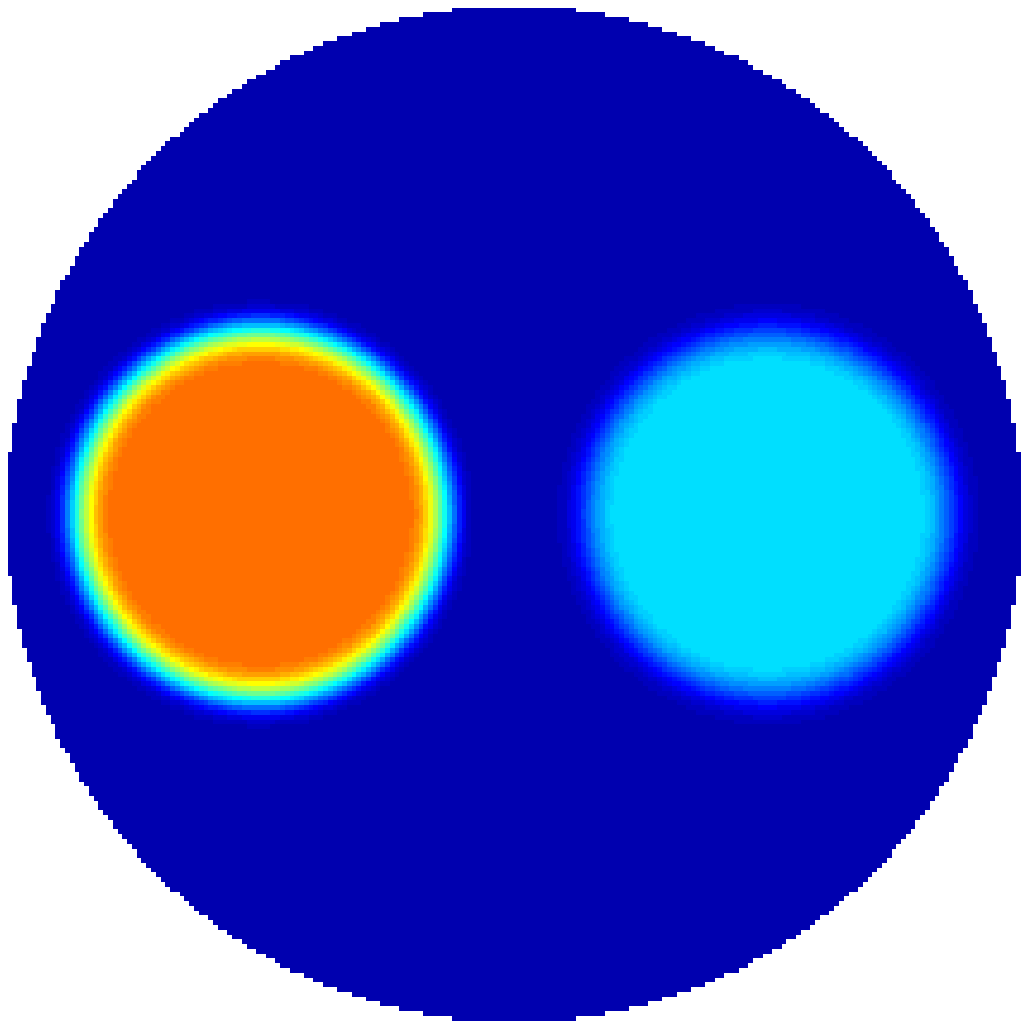}}
\put(140,0){\includegraphics[height=130pt]{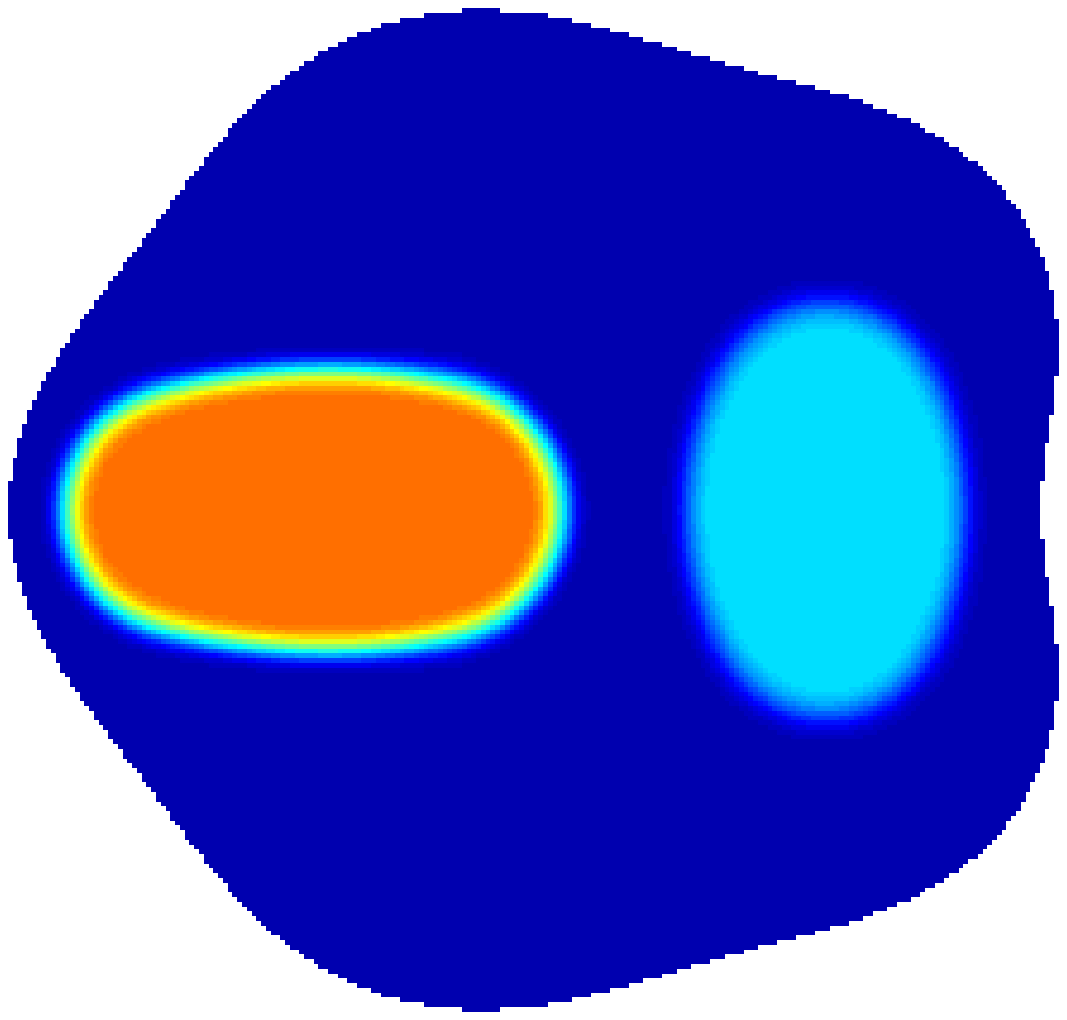}}
\put(280,0){\includegraphics[height=130pt]{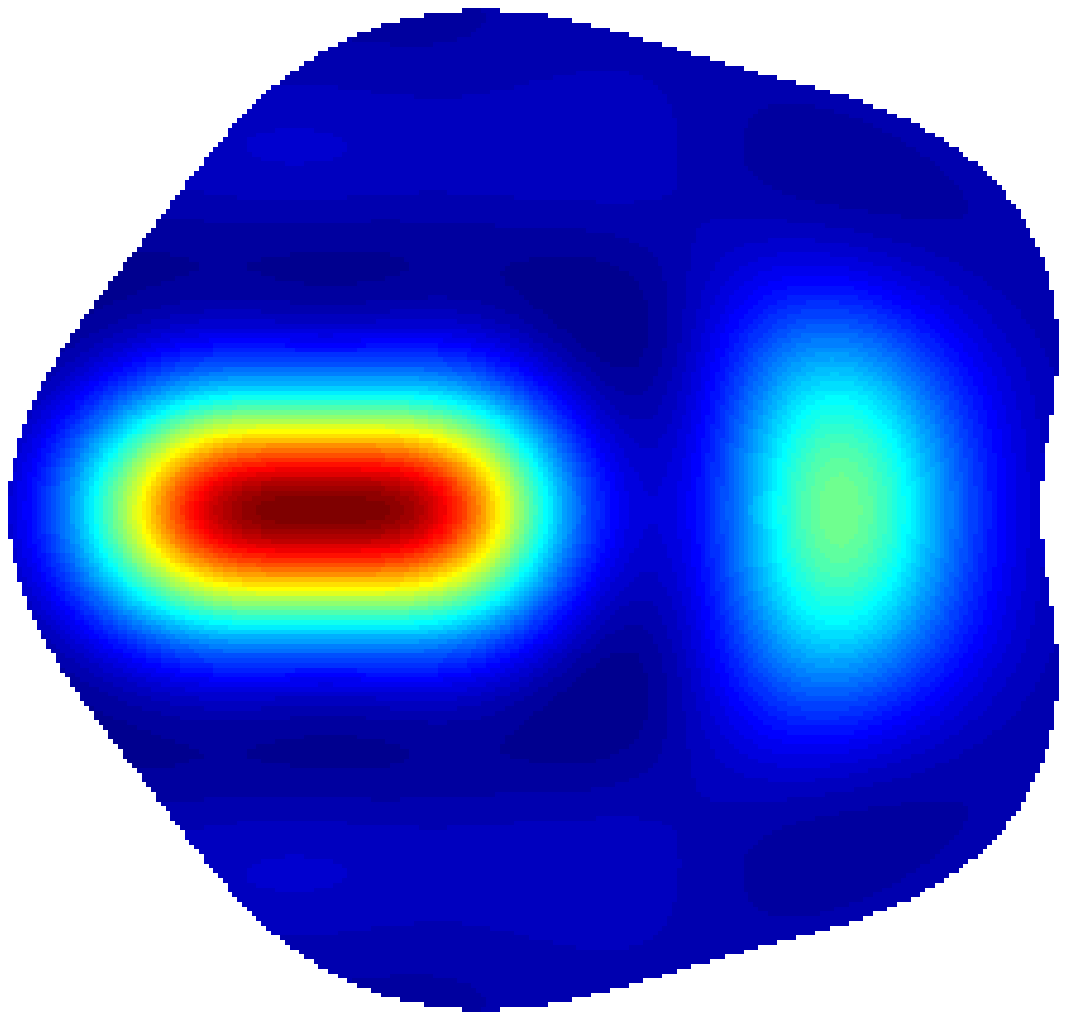}}
\put(35,130){${\sqrt{\det\sigma(z)}}$}
\put(175,130){${\sqrt{\det\sigma(\zeta)}}$}
\put(330,130){${\gamma_R(\zeta)}$}
\end{picture}
\caption{\label{fig-recon-PIPES-c2} Results for the $C^2$ smooth circular inclusions phantom of Figure~\ref{fig-phantoms-PIPES} in Test~1.  Left: Isotropization of the true conductivity $\sigma$ shown in the physical coordinates.  Middle: Isotropization of the true conductivity $\sigma$ shown in the deformed coordinates.  Right: Reconstruction of the isotropization of $\sigma$ in the deformed coordinates.   The scattering transform $\mathbf{t}_R(k)$ was computed for $|k|\leq 6$.  The values of the reconstruction $\gamma_R$ are very similar to the true values (Left and Middle), and the change of coordinates is evident from the squeezed ellipses.  The figures are plotted on the same color scale for ease of comparison.} 
\end{figure}
The anisotropy is visible in the isotropic representation $\gamma_R$ through the coordinate deformation (circles squeezed into ellipses). This is due to the precise underlying quasiconformal map $F$ which satisfies
\[\sqrt{\det\sigma}\left(\begin{array}{cc}
1 & 0 \\
0 & 1
\end{array}\right) = \frac{1}{ J_F}\;DF\;\sigma \left(DF\right)^t\]
where $DF$ denotes the Jacobian of the map $F$, $J_F$ the determinant of the Jacobian $DF$, and $t$ the transpose.  For $z$ inside the inclusion $\sigma_1(z)=\left(\begin{array}{cc}
1 & 0 \\
0 & 4
\end{array}\right)$ the map will deform the coordinates (approximately) by $F:z=(x,y)\mapsto \left(x,\frac{1}{2}y\right)$.  Similarly for $z\in\sigma_2(z)=\left(\begin{array}{cc}
2 & 0 \\
0 & 1
\end{array}\right)$, the map will behave (approximately) as $F:z=(x,y)\mapsto \left(\frac{1}{\sqrt{2}}x,y\right)$ which provides intuition for the vertical and horizontal squeezing displayed in the reconstructions of Figure~\ref{fig-recon-HnL-c2}.  

\subsection{Test~2: A Heart and Lungs Phantom}
Our next test involves the anisotropic phantoms shown in Figure~\ref{fig-phantoms-HnL} with two ovular inclusions (lungs) and a circular inclusion (heart) with conductive directional preferences
\begin{equation}\label{eq-sig-HnL-sig3-sig4}
\sigma_3=\left[\begin{array}{cc}
0.4 & 0\\
0 &0.8\\
\end{array}\right],\qquad
\sigma_4=\left[\begin{array}{cc}
6 & 0\\
0 &2\\
\end{array}\right].
\end{equation}
The ovular lung inclusions are less conductive than the background with a directional preference in the $x_2$ direction whereas the heart shaped inclusion more conductive in the  $x_1$ direction.  Here we consider a piecewise constant phantom and its $C^2$ smoothed version (right and left in Figure~\ref{fig-phantoms-HnL}, respectively).  Clearly this is an idealized version of a 2D cross-section and does not model the true 3D nature of the anistropy present in the heart.

\begin{figure}[h!] 
\centering
\begin{picture}(285,130)
\put(0,-5){\includegraphics[height=125pt]{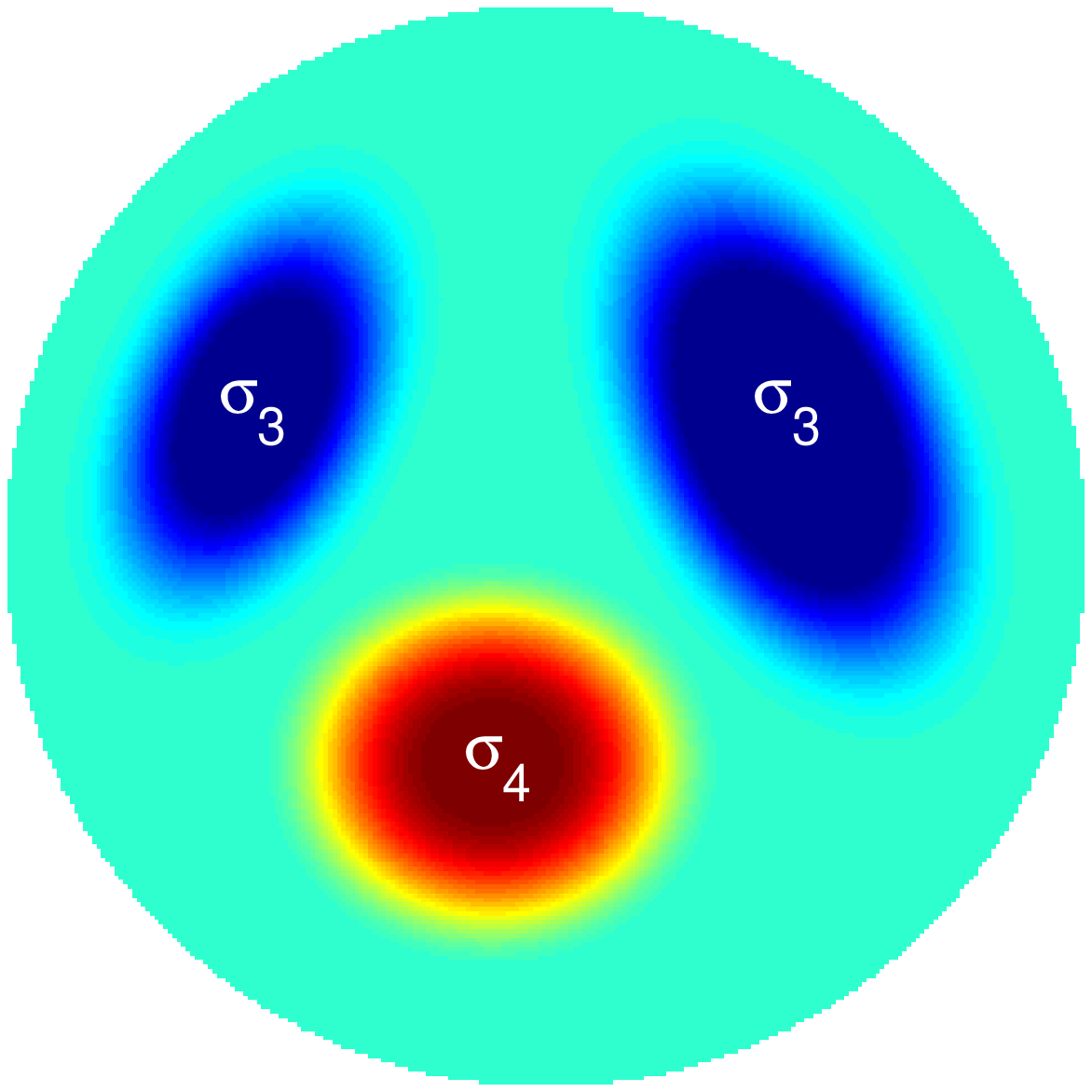}}
\put(150,-5){\includegraphics[height=125pt]{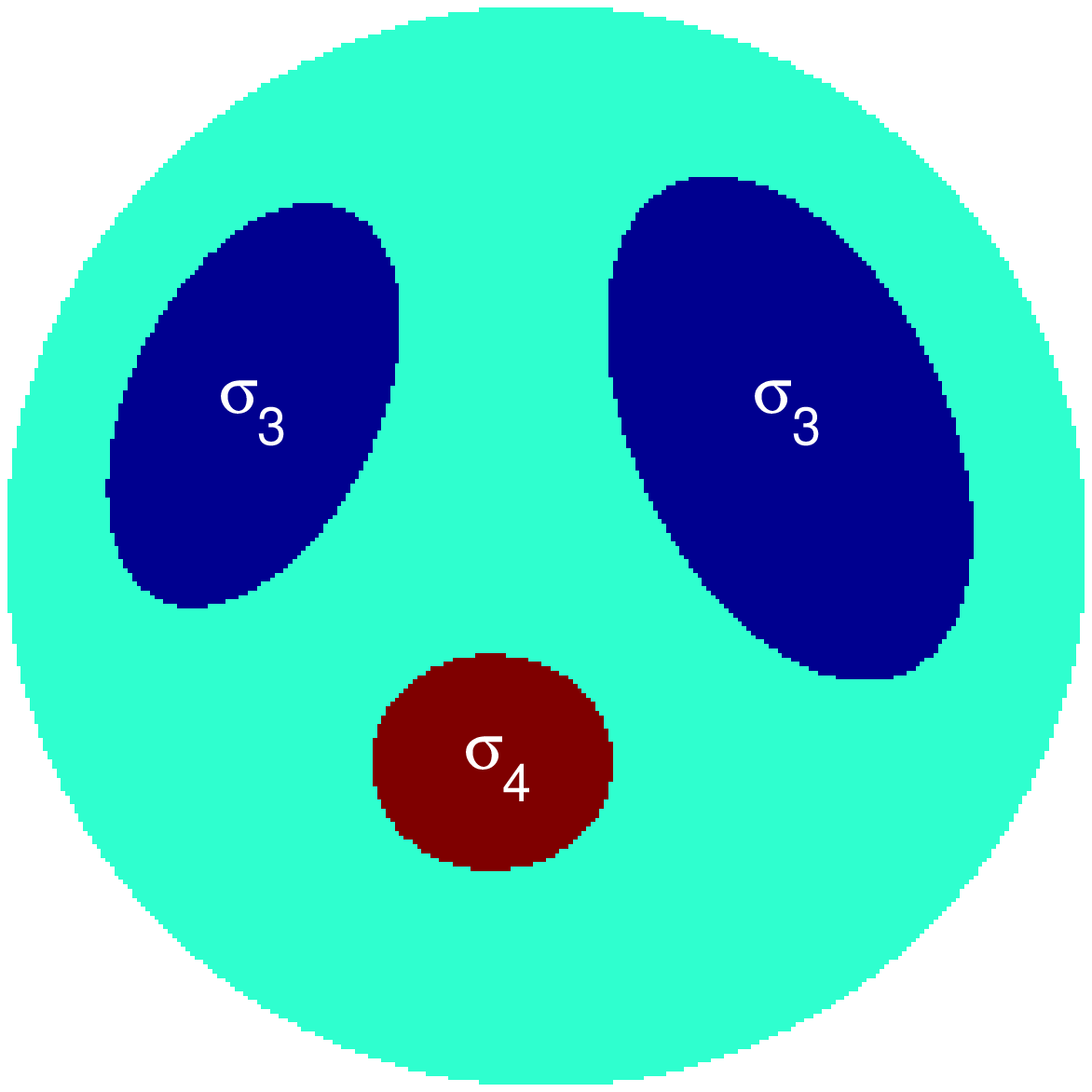}}
\put(30,122){\textbf{$C^2$ smooth}}
\put(170,122){\textbf{Discontinuous}}
\end{picture}
\caption{\label{fig-phantoms-HnL}The heart and lungs phantom used in Test~2.  The $C^2$ smooth and corresponding piecewise constant phantoms are shown on the left and right, respectively.  The conductivity in the ovular lungs regions, $\sigma_3$,  and in the circular heart region $\sigma_4$ are defined in \eqref{eq-sig-HnL-sig3-sig4}. The background conductivity is the $2\by 2$ identity matrix.}
\end{figure}

We first work with the $C^2$ smoothed  phantom shown in Figure~\ref{fig-phantoms-HnL}.  The boundary integral equation \eqref{AP_BIEmatrix} was solved to recover $M^\pm_\sigma$ for $z\in\p\D$ as described above, with $|k|\leq 7$ to evaluate the scattering transform $\mathbf{t}_R(k)$ on a $k$-grid with a stepsize of $h_k\approx 0.109$.  The $\dbark$-equation was then solved and the isotropic representative conductivity $\gamma_R(\zeta)\approx(F_*\sigma)(\zeta)=\sqrt{\det\sigma\left(F^{-1}(\zeta)\right)}$ was recovered.  Figure~\ref{fig-recon-HnL-c2} shows the reconstructed $\gamma_R(\zeta)$ along with $\sqrt{\det\sigma(\zeta)}$.  The maximum value in $\sqrt{\det \sigma}$ occurs in the \emph{heart} with a value of 3.46  and the minimum in the \emph{lungs} with a value of 0.57.  The recovered values are 4.42 and 0.50, respectively. 
The change of coordinates is evident as the lungs are squeezed vertically and the heart horizontally.
\begin{figure}[h!] 
\centering
\begin{picture}(410,135)
\put(0,0){\includegraphics[height=130pt]{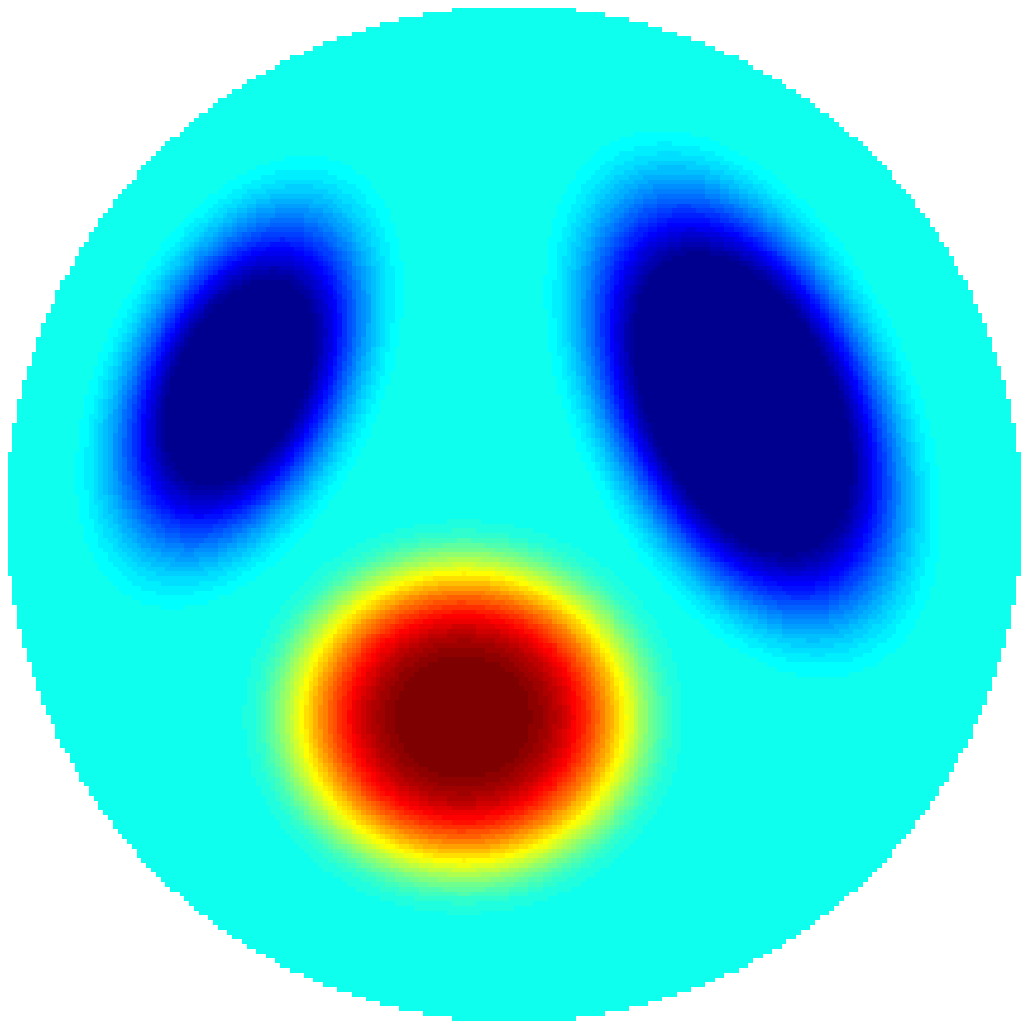}}
\put(140,0){\includegraphics[height=130pt]{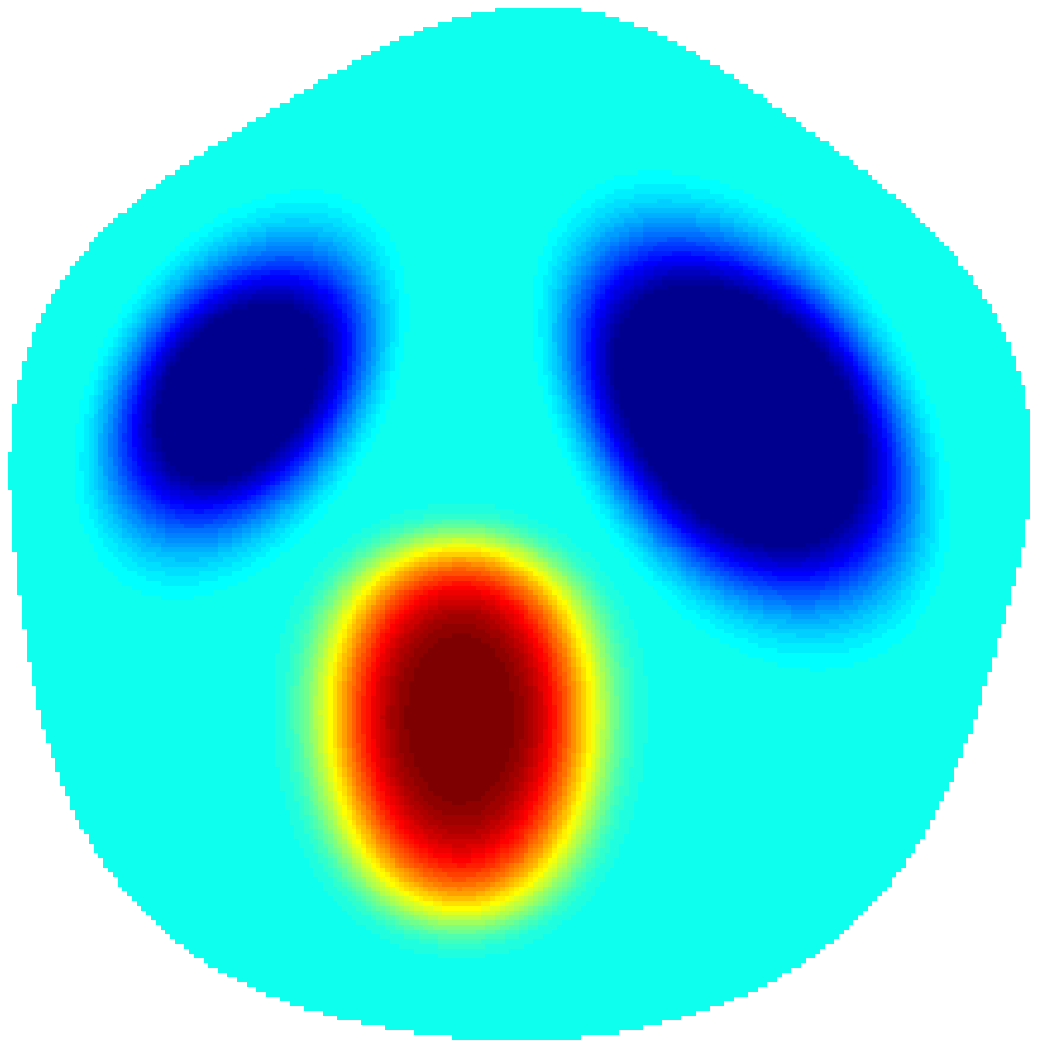}}
\put(280,0){\includegraphics[height=130pt]{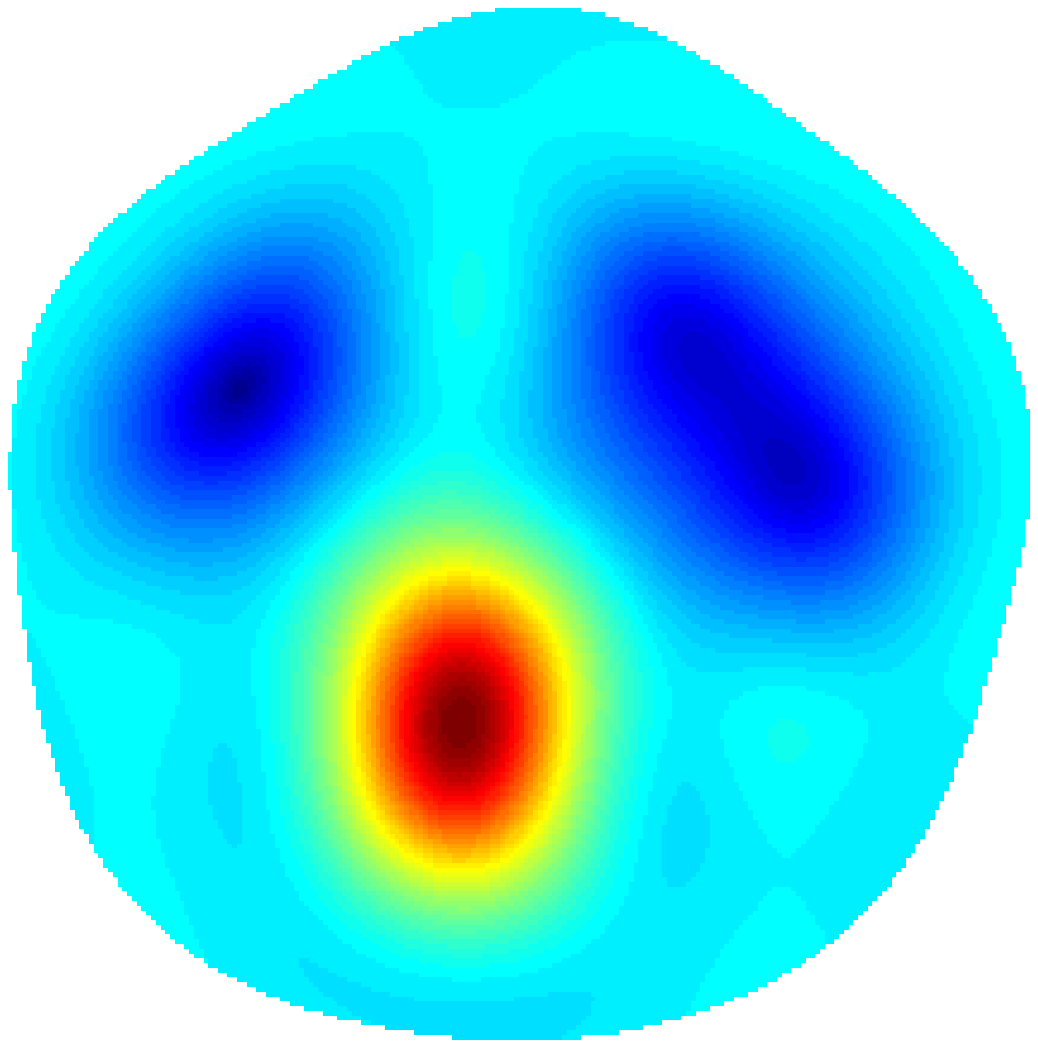}}
\put(35,130){${\sqrt{\det\sigma(z)}}$}
\put(175,130){${\sqrt{\det\sigma(\zeta)}}$}
\put(330,130){${\gamma_R(\zeta)}$}
\end{picture}
\caption{\label{fig-recon-HnL-c2} Results for the $C^2$ smooth heart and lungs phantom in Figure~\ref{fig-phantoms-HnL} in Test~2.  Left: Isotropization of the true conductivity $\sigma$ shown in the physical coordinates.  Middle: Isotropization of the true conductivity $\sigma$ shown in the deformed coordinates.  Right: Reconstruction of the isotropization of $\sigma$ in the deformed coordinates.   The scattering transform $\mathbf{t}_R(k)$ was computed for $|k|\leq 7$.  The squeezing in the isotropic reconstruction $\gamma_R$ clearly shows stronger $x_1$ and $x_2$ anisotropies in the heart and lung inclusions, respectively.  The figures are plotted on the same color scale for ease of comparison.}
\end{figure}

\begin{figure}[h!] 
\centering
\begin{picture}(410,145)
\put(0,0){\includegraphics[height=130pt]{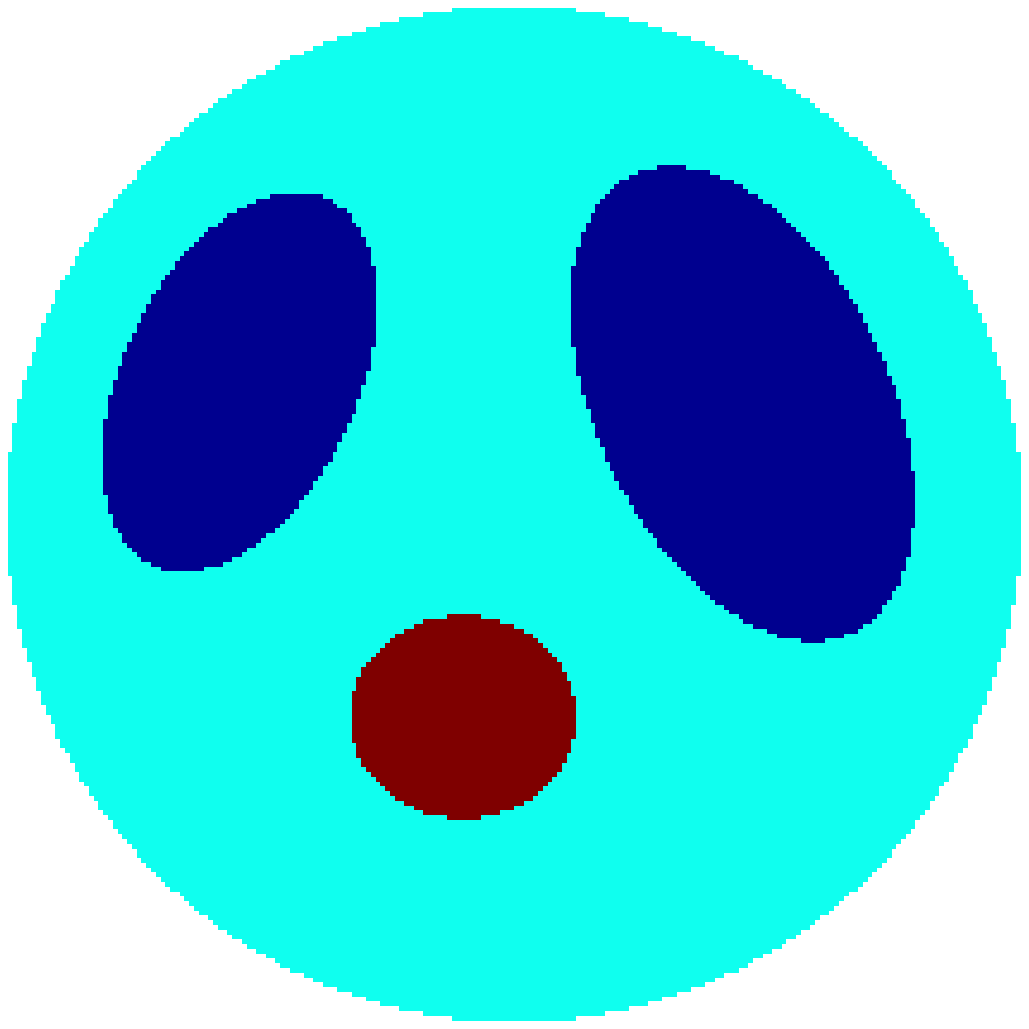}}
\put(140,0){\includegraphics[height=130pt]{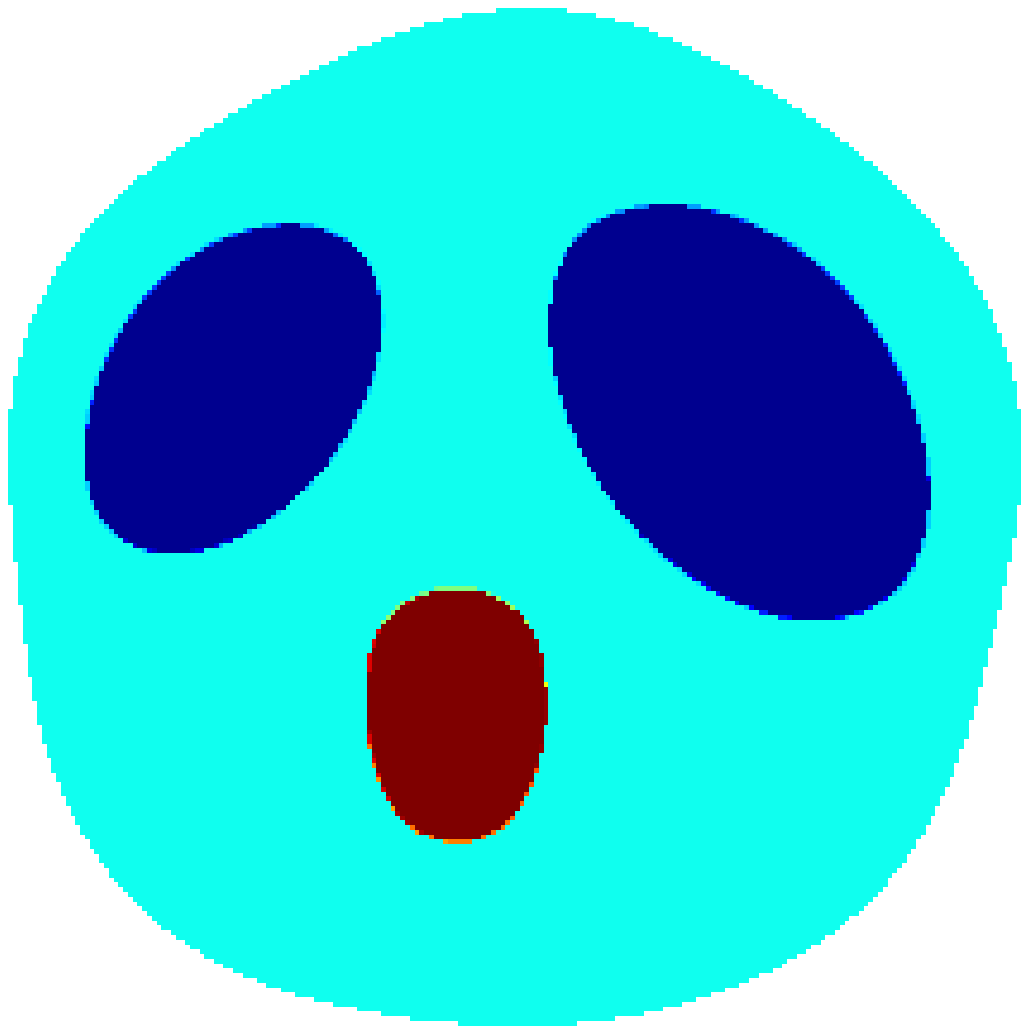}}
\put(280,0){\includegraphics[height=130pt]{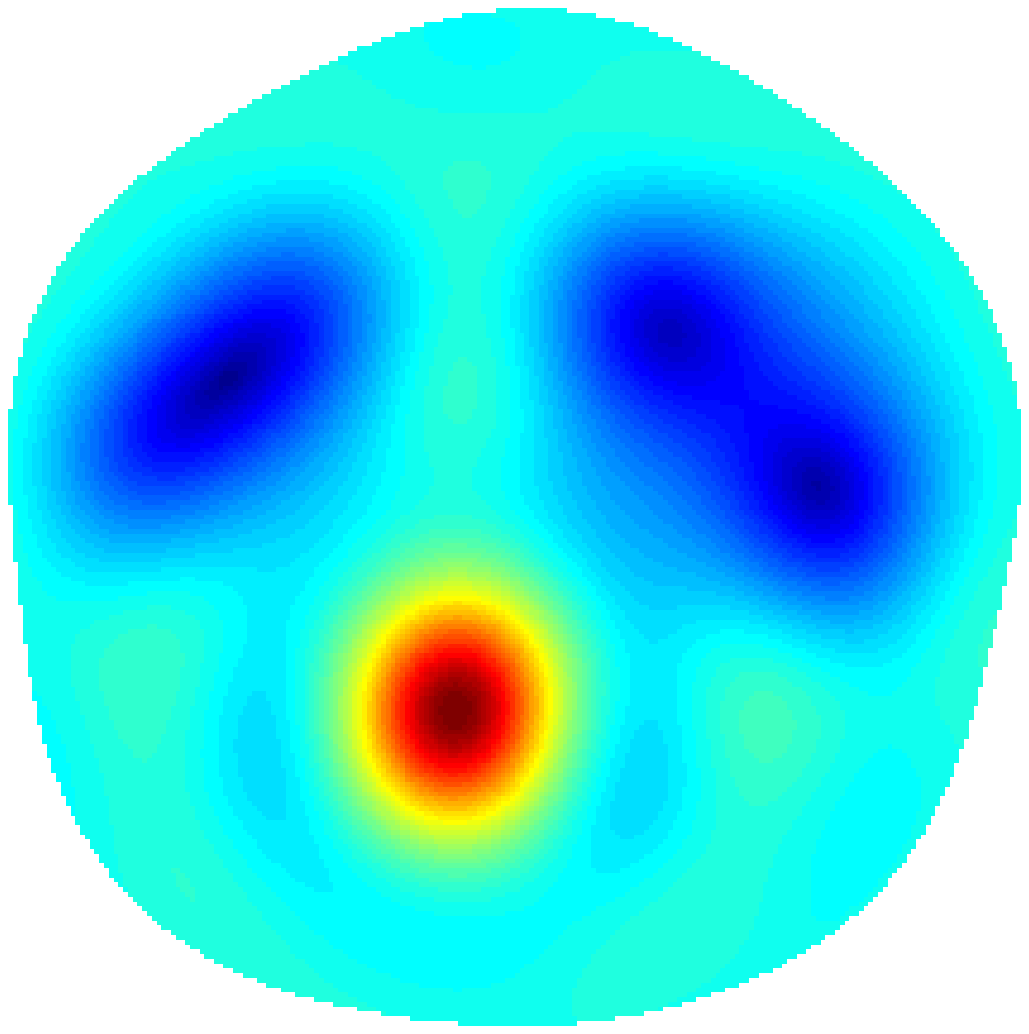}}
\put(35,130){${\sqrt{\det\sigma(z)}}$}
\put(175,130){${\sqrt{\det\sigma(\zeta)}}$}
\put(330,130){${\gamma_R(\zeta)}$}
\end{picture}
\caption{\label{fig-recon-HnL-Linf} Results for the piecewise constant heart and lungs phantom in Figure~\ref{fig-phantoms-HnL} in Test~2.  Left: Isotropization of the true conductivity $\sigma$ shown in the physical coordinates.  Middle: Isotropization of the true conductivity $\sigma$ shown in the deformed coordinates.  Right: Reconstruction of the isotropization of $\sigma$ in the deformed coordinates.   The scattering transform $\mathbf{t}_R(k)$ was computed for $|k|\leq 7$.  Again, the squeezing in the isotropic reconstruction $\gamma_R$ clearly shows stronger $x_1$ and $x_2$ anisotropies in the heart and lung inclusions, respectively.  The figures are plotted on the same color scale for ease of comparison.}
\end{figure}

It should be noted that while the theory described above does not theoretically hold for discontinuous phantoms, as we passed to the $C^2$ smooth proof of Nachman \cite{Nachman1996}, the D-bar algorithm resulting from Nachman's Schr\"odinger based $C^2$ proof has been used effectively on piecewise constant phantoms suggesting that the Schr\"odinger based D-bar algorithm can be used on discontinuous phantoms.  In light of this, we tested the algorithm on the discontinuous version of the $C^2$ phantom also shown in Figure~\ref{fig-phantoms-HnL}.  

The same $k$ and $z$ grids were used as in the $C^2$ smoothed case.  Figure~\ref{fig-recon-HnL-Linf} shows the reconstructed $\gamma_R$ again with strong $x_1$ anisotropy in the heart and $x_2$ in the lungs clearly visible via the squeezing of the underlying change of coordinates.  Here the maximum and minimum recovered values are are 4.03 and 0.47, respectively. 
\subsection{Test~3:  Increased Noise}
We now revisit the $C^2$-smooth phantom in Test~1, and introduce Gaussian relative noise to the simulated FEM voltage data as follows.  Let $V^j$ denote the vector of computed boundary voltages for the $j$-th current pattern, $\eta$ the noise level, and $N^j$ a Gaussian random vector (generated by the \texttt{randn} command in {\sc MATLAB}) that is unique for each current pattern $j$.  Denote by $\widetilde{V}^j$ the noisy voltage data computed by
\[\widetilde{V}^j= V^j + \eta N^j \max\left| V^j\right|,\quad 1\leq j\leq 32.\]
The corresponding noisy N-D matrix $\widetilde{\mathbf{R}}_\sigma$ was then computed using $\widetilde{R}_\sigma\phi_j=\widetilde{V}^j$.  The noisy D-N matrix $\widetilde{L}_\sigma$ was then formed using $\widetilde{\mathbf{R}}_\sigma$.

\begin{figure}[h!]
\centering
\hspace{-3.5em}
\begin{picture}(415,110)
\put(0,0){\includegraphics[height=90pt]{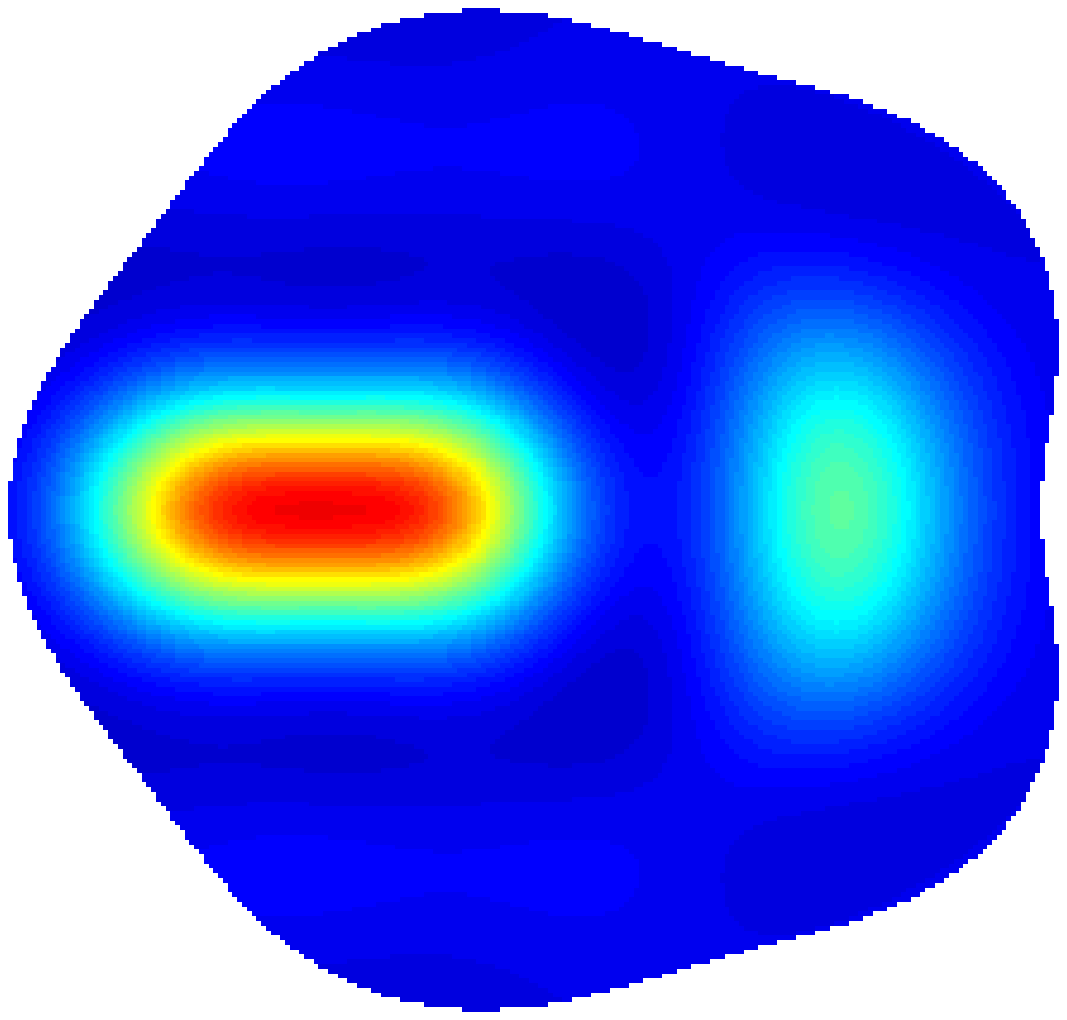}}
\put(90,0){\includegraphics[height=90pt]{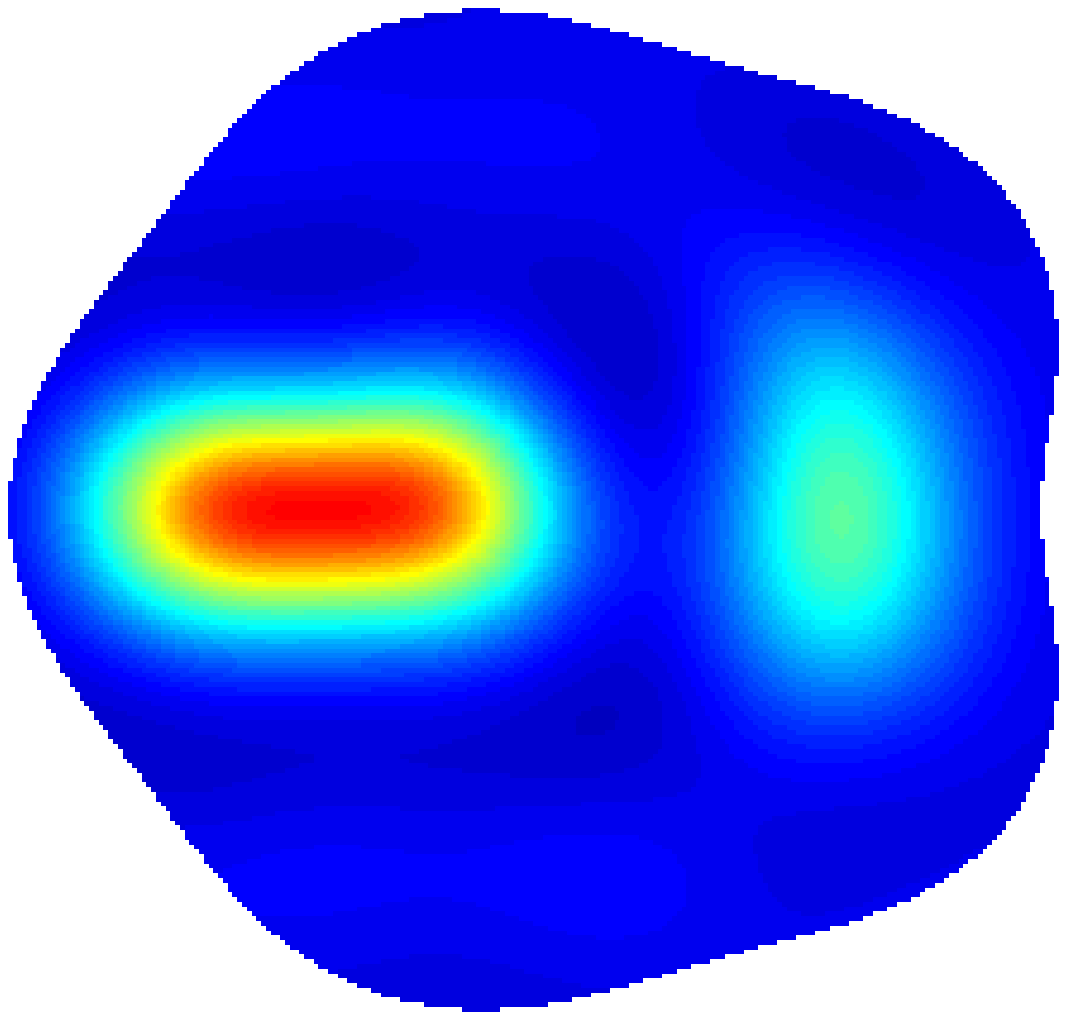}}
\put(180,0){\includegraphics[height=90pt]{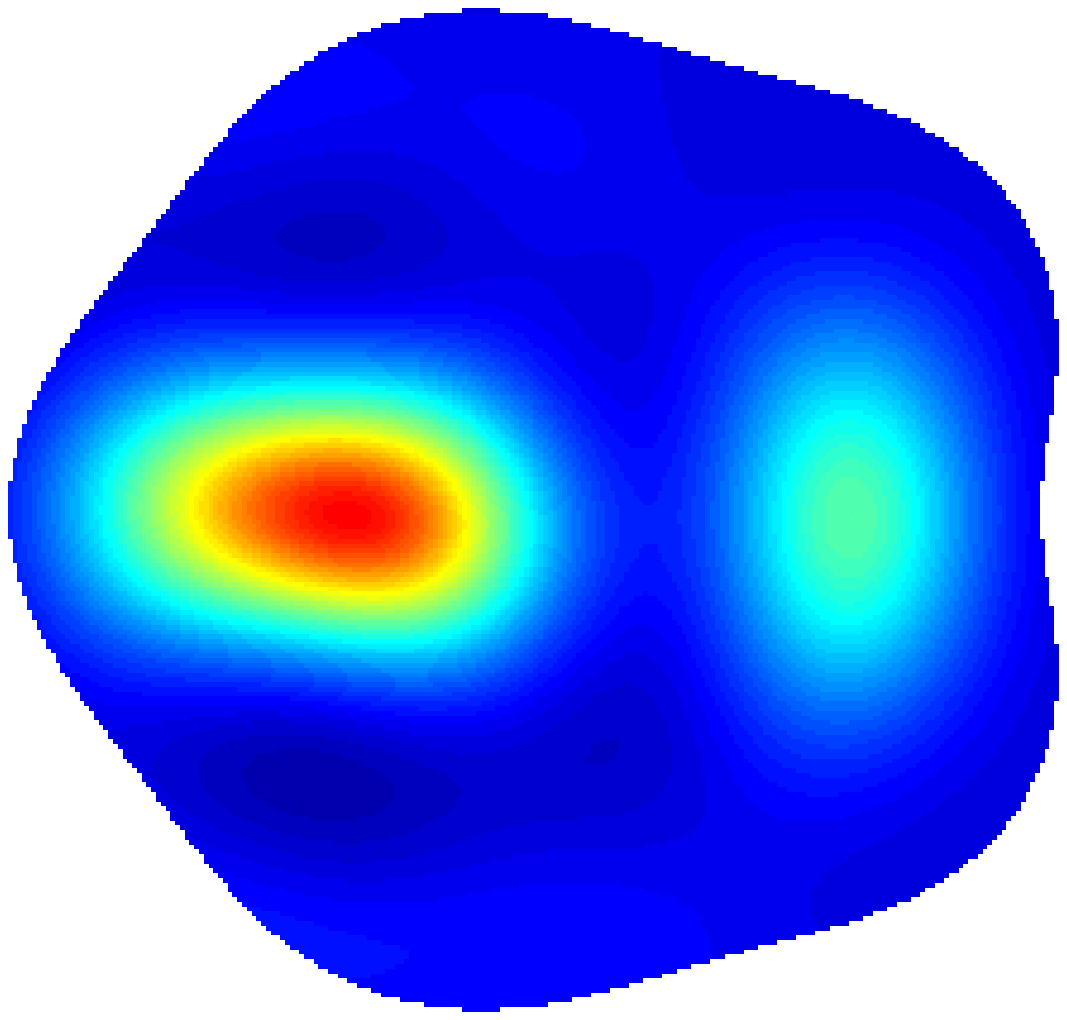}}
\put(270,0){\includegraphics[height=90pt]{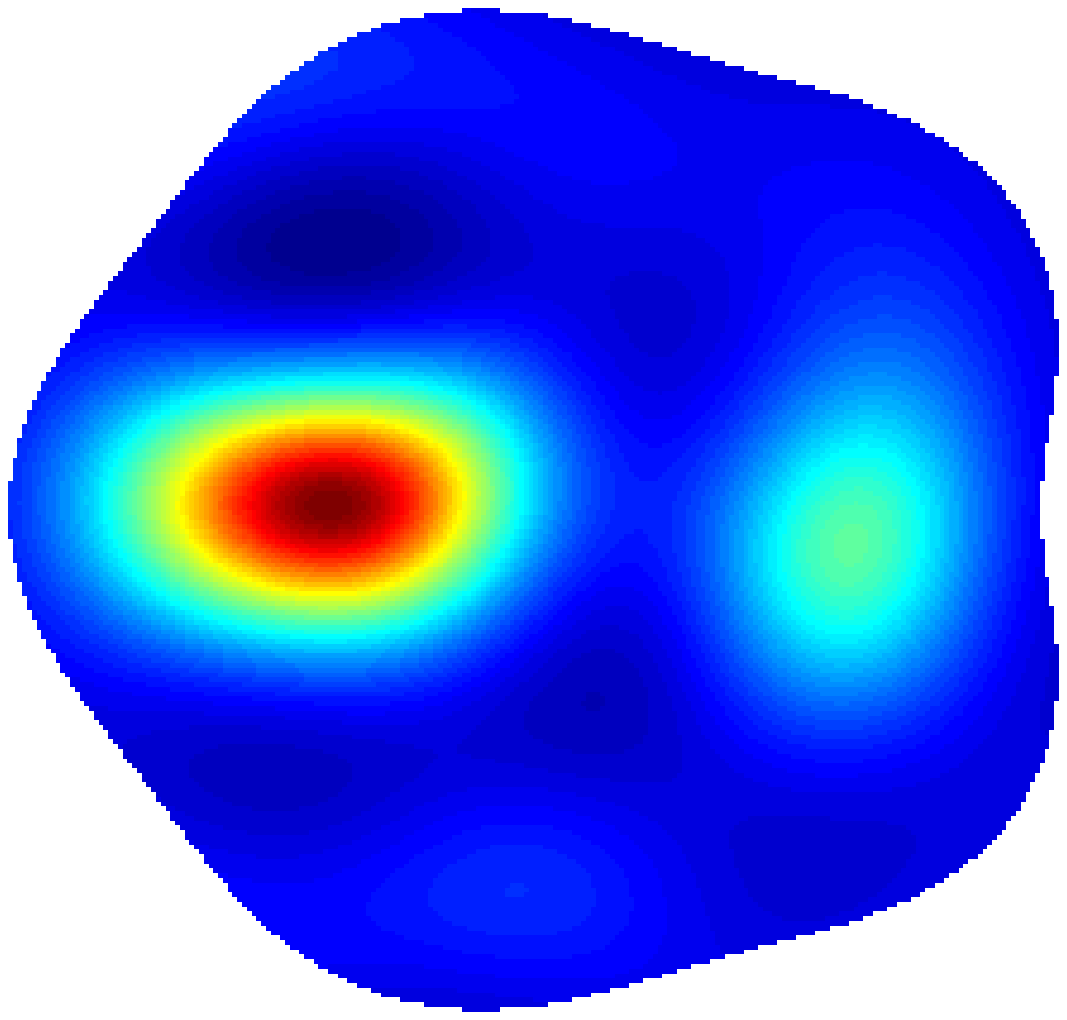}}
\put(360,0){\includegraphics[height=90pt]{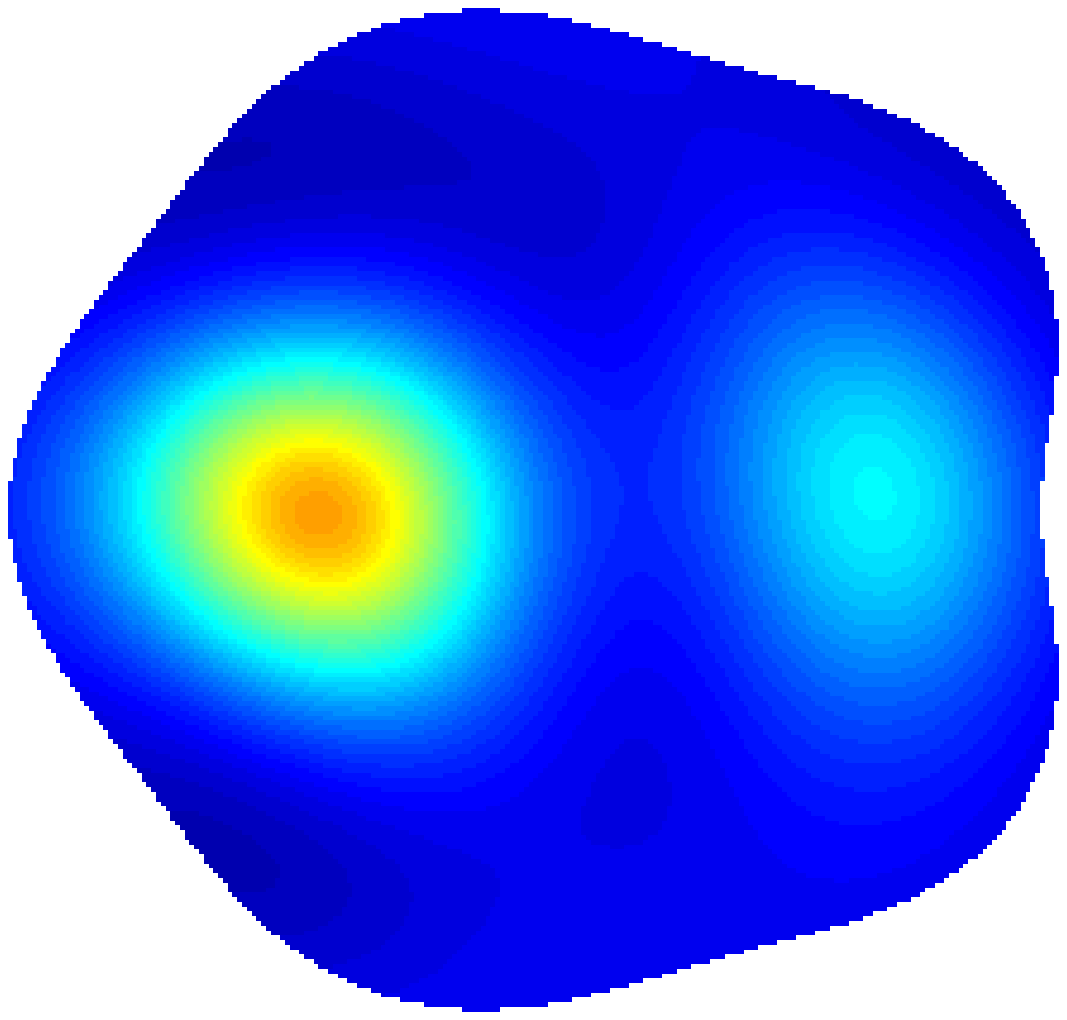}}
\put(35,92){$0\%$}
\put(120,92){$0.01\%$}
\put(210,92){$0.10\%$}
\put(300,92){$0.25\%$}
\put(390,92){$1.0\%$}
\end{picture}
\caption{\label{fig-recon-PIPES-c2-NOISY-volts} Reconstructions of the $C^2$ smooth circular inclusions phantom of Figure~\ref{fig-phantoms-PIPES} in Test~1 with various levels of additional noise added to the voltage data.  The scattering data for the reconstructions, from left to right, were computed for $|k|\leq 6.0$, 5.9, 5.0, 4.8, and 3.8.  The figures are plotted on the same color scale for ease of comparison.  Notice that as the level of noise increases, the approximate values of the associated isotropic conductivity $\sqrt{\det\sigma}$ are retained, however the directional preferences of the underlying anisotropy become less pronounced.}
\end{figure}

Figure~\ref{fig-recon-PIPES-c2-NOISY-volts} shows the reconstructions with 0\%, 0.01\%, 0.10\%, 0.25\%, and 1.0\%  Gaussian relative noise added to the voltage data.  As the level of noise increased, the radius for the scattering data was forced to decrease, i.e. from $|k|\leq 6.0$, 5.9, 5.0, 4.8, and 3.8.  Table~\ref{table-noisy} contains the maximum values of the left and right inclusions in the reconstructions for the various levels of noise.  It is clear that as the noise level increases it is still possible to determine the approximate values for the isotropic representative conductivity $\sqrt{\det \sigma}$, although the directional preference of the underlying anisotropy becomes less pronounced.  A more detailed study of various types of additional noise is outside the scope of this paper, see \cite{HM12_NonCirc} for a discussion of how noise levels are specific to the EIT device.

\begin{table}[h!]\label{table-noisy}
\centering 
\begin{tabular}{|l|c|c|c|c|c|c|}
  \hline
    &&&&&&\\
Noise Level & True & 0\% & 0.01\% &  0.10\% & 0.25\% & 1.0\% \\
    &&&&&&\\
    \hline
    &&&&&&\\
Left inclusion & 2.00& 2.34 & 2.31 & 2.32 & 2.54 & 2.05 \\
    &&&&&&\\
Right inclusion & 1.41& 1.61 & 1.60& 1.61 & 1.55 & 1.45\\
      \hline
\end{tabular}\caption{Maximum values in the reconstructed isotropizations shown in Figure~\ref{fig-recon-PIPES-c2-NOISY-volts}.}
\end{table}

\subsection{The Deformed Boundary}
It should be noted that the representative isotropic conductivities $\gamma_R$ displayed in this paper are shown on their true deformed boundary (calculated by solving \eqref{eq-Beltrami-F} for each known anisotropic $\sigma$).  In practice, this information is not readily available.  However, Figure~\ref{fig-bndries-def} shows the deformations of the examples, from Tests~1 and 2, compared to the original boundary ($\p\D$) and a circle of radius 1.2.  Notice that the deformation is clearly contained inside the disc of radius 1.2, suggesting that recovering the anisotropic conductivity on a slightly larger domain is sufficient to contain the interior behavior.  Furthermore, as the D-bar method used here allows for point-wise reconstructions for a given point $\zeta$ and it is known that $\gamma\equiv 1$ outside $\widetilde{\Omega}$, one can merely compute at additional $\zeta$ points outside of the original domain $\Omega=\D$ to ensure that the deformed domain is in fact contained.  The reconstruction $\gamma_R$ should be interpreted as a reconstruction in isothermal coordinates, similar to sonograms in ultrasound imaging. 

\begin{figure}[h!] 
\centering
\begin{picture}(400,155)
\put(0,0){\includegraphics[height=125pt]{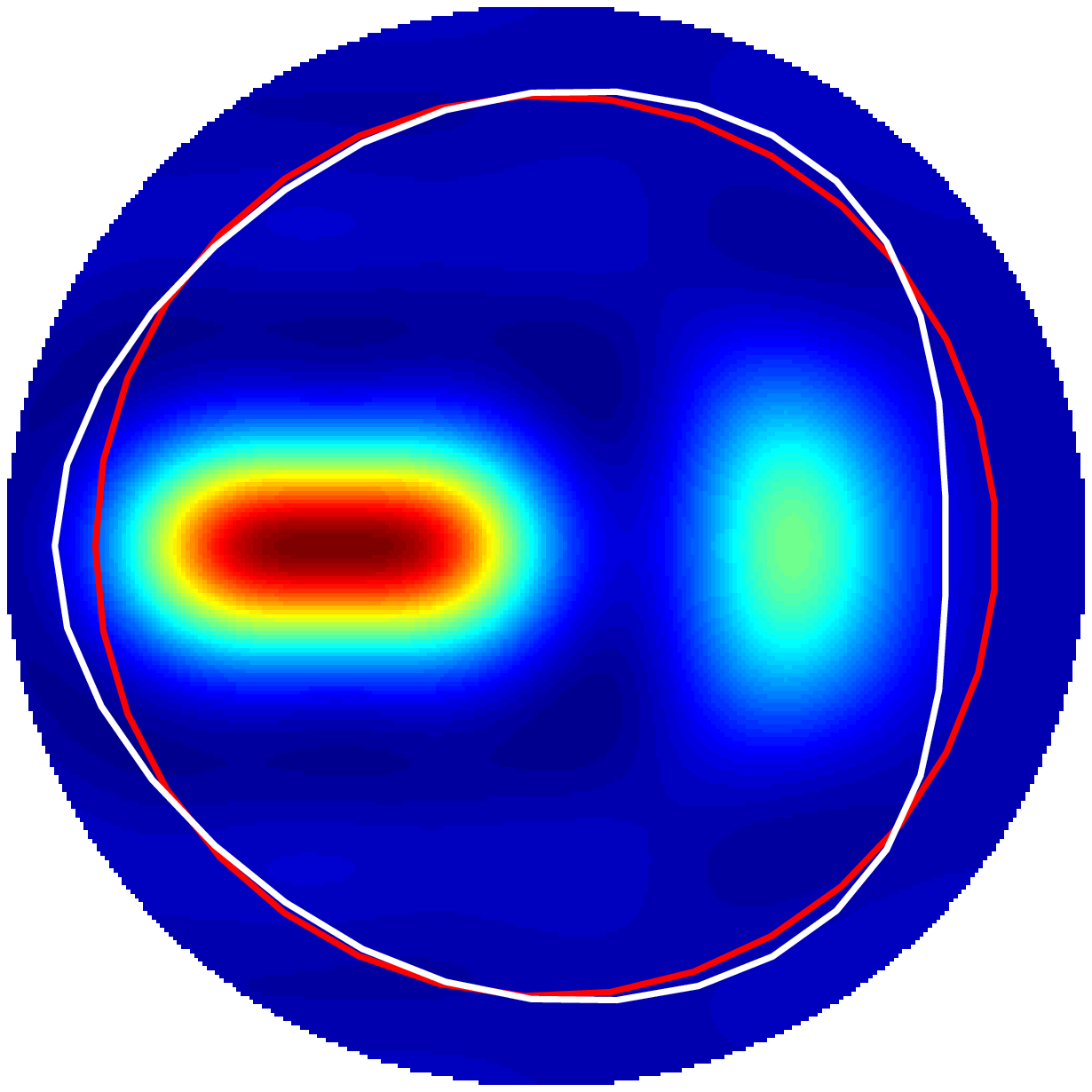}}
\put(140,0){\includegraphics[height=125pt]{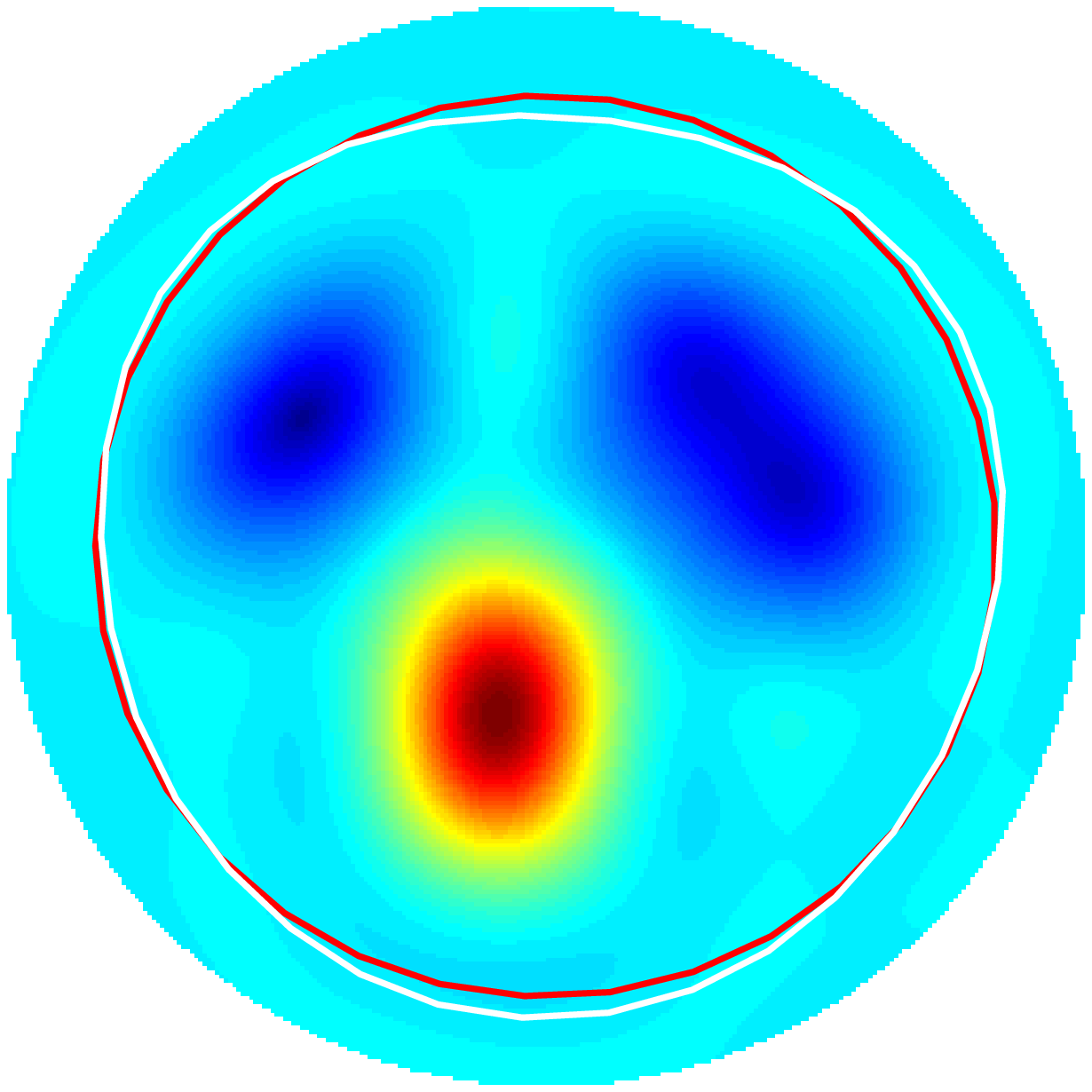}}
\put(280,0){\includegraphics[height=125pt]{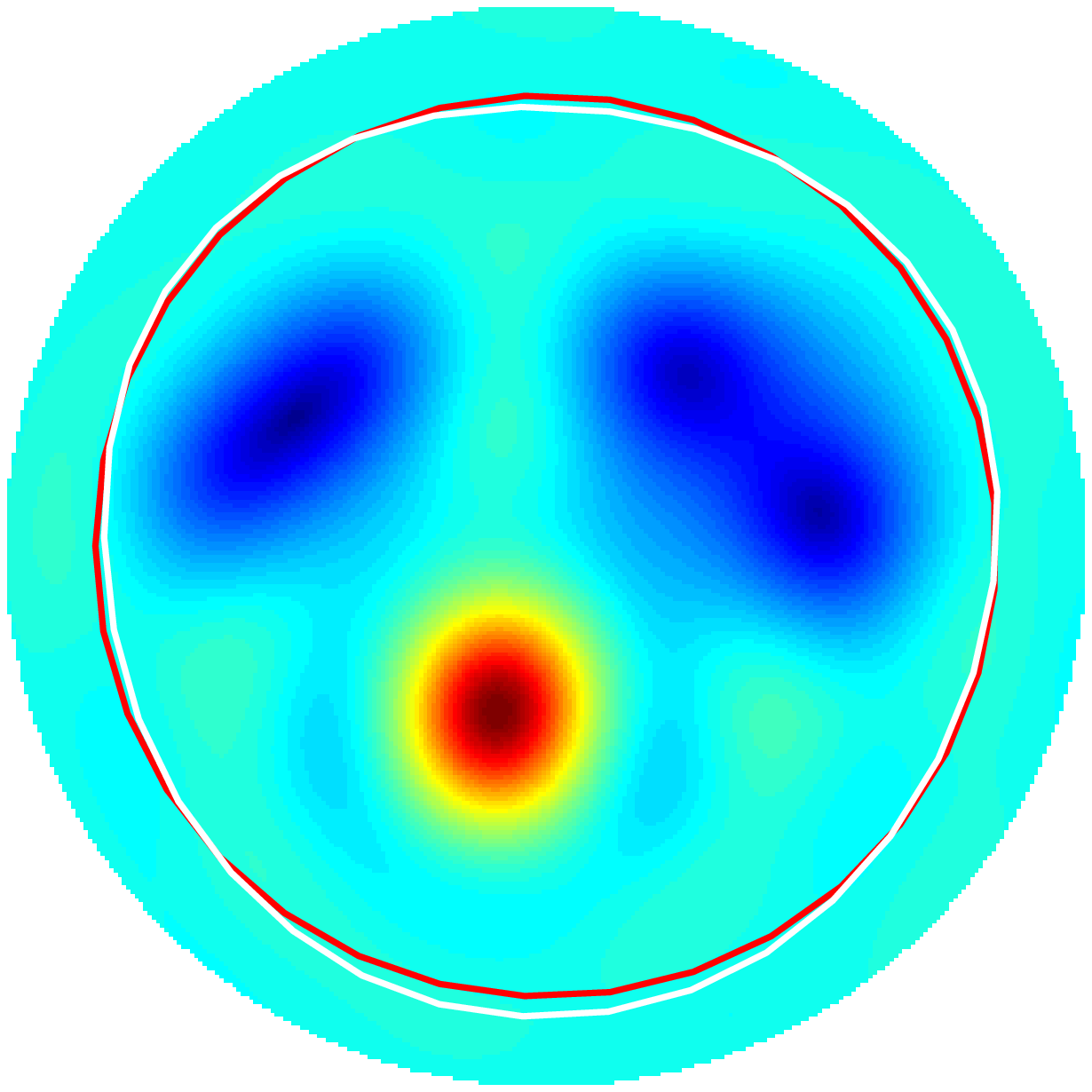}}
\put(40,145){\textbf{Test 1}}
\put(30,130){\textbf{$C^2$ smooth}}

\put(180,145){\textbf{Test 2}}
\put(170,130){\textbf{$C^2$ smooth}}

\put(320,145){\textbf{Test 2}}
\put(300,130){\textbf{Discontinuous}}
\end{picture}
\caption{\label{fig-bndries-def} Reconstructions of Tests~1 and 2 for $|\zeta|<1.2$. The red line indicates the unit circle, the original boundary $\p\D$, whereas the white line represents the deformed boundary calculated via \eqref{eq-Beltrami-F}.  Notice that the deformation is clearly contained inside of the larger domain.}
\end{figure}

We remind the reader that alternatively one could attempt to construct the precise deformation of the boundary $F(\partial\D)=\p\widetilde{\Omega}$.  However, this requires the construction of the highly unstable map $F$ \cite{Astala2005} via a high-frequency $k$ limit of the {\sc CGO} solutions \eqref{eq-Fz-lim}, for $z\in\partial\D$
\[F(z)=\lim_{|k|\to\infty} \frac{\log W^+_\sigma(z,k)}{ik}, \quad z\in\C\setminus\D,\]
where the logarithm is not necessarily taken with respect to the principal branch.  The traces of the CGO solutions $W^+_\sigma(z,k)$ are highly unstable for large $k$ making such a limit impractical numerically.  An in-depth analysis of the stability of such an endeavor is outside the scope of this paper.  

\section*{Conclusions}
In this paper, we have presented a constructive \textsc{CGO} based D-bar proof for $C^2$ smooth anisotropic conductivities in the plane which does not require the construction of the unstable \emph{push-forward} map $F$ used previously in the $L^\infty$ proof of \cite{Astala2005}, and previous {\sc CGO} based 2D anisotropic proofs.  The new proof, presented here, results in the first noise-robust nonlinear D-bar algorithm for the anisotropic conductivity problem in 2D {\sc EIT}.  We have demonstrated that the numerical D-bar algorithm performs well on simulated anisotropic EIT data and noise-robust images of isotropic, or scalar-valued, conductivities that are distorted versions of the original anisotropic conductivities, can be reliably recovered.  

We emphasize that non-uniqueness of the anisotropic inverse conductivity problem should not be such a game-ender and remind the reader that even the isotropic inverse conductivity problem is so severely ill-posed that strong regularization is needed for robust image formation from noisy data.  In a similar fashion to linear inverse problems, non-uniqueness can be dealt with by picking out a unique representative from the same-data equivalence class, which can still provide useful information.  In addition, when the  proposed method is applied to isotropic data, the result coincides with the traditional isotropic D-bar method.

The results of this paper additionally help to explain the geometric distortion observed in isotropic D-bar imaging when there is insufficient knowledge of the boundary, or when the underlying conductivity is in fact anisotropic.  Furthermore, the approach used here paves the way for direct regularized EIT for matrix-valued conductivities, one of the most important aspects of any practical inverse problem, and essential for ensuring a noise-robust reconstruction algorithm. 

\section*{Acknowledgments}
This study was supported by the SalWe Research Program for Mind and Body (Tekes - the Finnish Funding Agency for Technology and Innovation grant 1104/10) and by the Academy of Finland (\textit{Finnish Centre of Excellence in Inverse Problems Research} 2012-2017, decision number 250215).
\bibliographystyle{siam}
\bibliography{bibliographyRefs_new}

\end{document}